\documentclass{article}
\usepackage{amsmath}
\usepackage{amsfonts}
\usepackage{amssymb}
\usepackage{graphicx}

\newtheorem{theorem}{Theorem}[section]

\newtheorem{lemma}[theorem]{Lemma}

\newtheorem{remark}[theorem]{Remark}

\newenvironment{proof}[1][Proof]{\noindent\textbf{#1.} }{\ \rule{0.5em}{0.5em}}
\numberwithin{equation}{section}
\input{tcilatex}

\begin{document}

\title{A-priori gradient bound for elliptic systems under either slow or
fast growth conditions}
\author{Tommaso Di Marco \& Paolo Marcellini \\
{\small Dipartimento di Matematica e Informatica "Ulisse Dini", Universit%
\`{a} di Firenze, Italy}}
\date{}
\maketitle

\begin{abstract}
We obtain an \textit{a-priori} $W_{\limfunc{loc}}^{1,\infty }\left( \Omega ;%
\mathbb{R}^{m}\right) -$bound for solutions in $\Omega \subset \mathbb{R}%
^{n} $, $n\geq 2$, to the elliptic system 
\begin{equation*}
\sum_{i=1}^{n}\frac{\partial }{\partial x_{i}}\left( \frac{g_{t}\left(
x,\left\vert Du\right\vert \right) }{\left\vert Du\right\vert }%
u_{x_{i}}^{\alpha }\right) =0,\;\;\;\;\;\alpha =1,2,\ldots ,m,
\end{equation*}%
where $g\left( x,t\right) $, $g:\Omega \times \left[ 0,\infty \right)
\rightarrow \left[ 0,\infty \right) $, is a Carath\'{e}odory function,
convex and increasing with respect to the gradient variable $t\in \left[
0,\infty \right) $. We allow $x-$dependence, which turns out to be a
relevant difference with respect to the autonomous case and not only a
technical perturbation. Our assumptions allow us to consider \textit{both
fast and slow growth}. We allow fast growth even of \textit{exponential type}%
; and slow growth, for instance of \textit{Orlicz-type} with
energy-integrands such as $g\left( x,\left\vert Du\right\vert \right)
=|Du|\log (1+|Du|)$ or, when $n=2,3$, even \textit{asymptotic linear growth}
with energy integrands of the type 
\begin{equation*}
g\left( x,\left\vert Du\right\vert \right) =\left\vert Du\right\vert
-a\left( x\right) \sqrt{\left\vert Du\right\vert }\,.
\end{equation*}
\end{abstract}

\section{Introduction\label{Section: Introduction}}

We are concerned with the regularity of \textit{local minimizers} of
energy-integrals of the calculus of variations of the form 
\begin{equation}
F\left( u\right) =\int_{\Omega }f\left( x,Du\right) \,dx\,,
\label{energy-integral 1}
\end{equation}%
where $\Omega $ is an open set of $\mathbb{R}^{n}$ for some $n\geq 2$ and $Du
$ is the $m\times n$ gradient-matrix of a map $u:\Omega \subset \mathbb{R}%
^{n}\rightarrow \mathbb{R}^{m}$ with $m\geq 1$. Here $f:\Omega \times 
\mathbb{R}^{m\times n}\rightarrow \mathbb{R}$ is a convex \textit{Carath\'{e}%
odory integrand}; i.e., $f=f\left( x,\xi \right) $ is measurable with
respect to $x\in \mathbb{R}^{n}$ and it is a convex function with respect to 
$\xi \in \mathbb{R}^{m\times n}$. A \textit{local minimizer} of the
energy-functional $F$ in (\ref{energy-integral 1}) is a map $u:\Omega
\subset \mathbb{R}^{n}\rightarrow \mathbb{R}^{m}$ satisfying the inequality 
\begin{equation*}
\int_{\Omega }f\left( x,Du\left( x\right) \right) \,dx\leq \int_{\Omega
}f\left( x,Du\left( x\right) +D\varphi \left( x\right) \right) \,dx
\end{equation*}%
for every test function $\varphi $ with compact support in $\Omega ;$ i.e., $%
\varphi \in C_{0}^{1}\left( \Omega ;\mathbb{R}^{m}\right) $. Under some 
\textit{growth conditions} on $f$ (see the proof of Lemma \ref{lemma1}\ for
details) every local minimizer $u$ of $F$ is a \textit{weak solution }to the 
\textit{nonlinear elliptic system} of $m$\textit{\ }partial differential
equations 
\begin{equation}
\sum_{i=1}^{n}\frac{\partial }{\partial x_{i}}a_{i}^{\alpha }\left(
x,Du\right) =0,\;\;\;\;\;\alpha =1,2,\ldots ,m,
\label{differential system 1}
\end{equation}%
where $a_{i}^{\alpha }\left( x,\xi \right) =\frac{\partial f}{\partial \xi
_{i}^{\alpha }}=f_{\xi _{i}^{\alpha }}$ for every $i=1,2,\ldots ,n$ and $%
\alpha =1,2,\ldots ,m$.

It is well known that in the \textit{vector-valued case} $m\geq 2$ in
general we cannot expect \textit{everywhere regularity} of local minimizers
of integrals as in (\ref{energy-integral 1}), or of weak solutions to
nonlinear differential systems as in (\ref{differential system 1}). Examples
of not smooth solutions are originally due to De Giorgi \cite{De Giorgi 1968}%
, Giusti-Miranda \cite{Giusti-Miranda 1968}, Ne\v{c}as \cite{Necas 1977},
and more recently to \v{S}ver\'{a}k-Yan \cite{Sverak-Yan 2000}, De
Silva-Savin \cite{DeSilva-Savin 2010}, Mooney-Savin \cite{Mooney-Savin 2016}%
, Mooney \cite{Mooney 2019}.

A classical assumption finalized to the \textit{everywhere regularity} is a
modulus-dependence in the energy integrand; i.e., in terms of the function $%
f $, we require that 
\begin{equation*}
f\left( x,\xi \right) =g\left( x,\left\vert \xi \right\vert \right)
\end{equation*}%
with a Carath\'{e}odory function $g=g\left( x,t\right) $. Since $\frac{%
\partial \left\vert \xi \right\vert }{\partial \xi _{i}^{\alpha }}=\frac{\xi
_{i}^{\alpha }}{\left\vert \xi \right\vert }$, the differential system
assumes the form 
\begin{equation}
\sum_{i=1}^{n}\frac{\partial }{\partial x_{i}}\left( \frac{g_{t}\left(
x,\left\vert Du\right\vert \right) }{\left\vert Du\right\vert }%
u_{x_{i}}^{\alpha }\right) =0,\;\;\;\;\;\alpha =1,2,\ldots ,m.
\label{differential system 2}
\end{equation}

In this \textit{nonlinear context}, the first regularity result is due to
Karen Uhlenbeck obtained in her celebrated paper \cite{Uhlenbeck 1977},
published in 1977 and related to the energy-integral $f\left( x,\xi \right)
=g\left( x,\left\vert \xi \right\vert \right) =\left\vert \xi \right\vert
^{p}$ with exponents $p\geq 2$. Later Marcellini \cite{Marcellini 1996} in
1996 considered general energy-integrands $g\left( \left\vert \xi
\right\vert \right) $ allowing \textit{exponential growth} and
Marcellini-Papi \cite{Marcellini-Papi 2006} in 2006 also some \textit{slow
growth}. Mascolo-Migliorini \cite{Mascolo-Migliorini 2003} studied some
cases of integrands $g\left( x,\left\vert \xi \right\vert \right) $ which
however ruled out the slow growth and power growth with exponents $p\in
\left( 1,2\right) $. Only recently Beck-Mingione introduced in the integrand
some $x-$dependence of the form $\int_{\Omega }\left\{ g\left( \left\vert
Du\right\vert \right) +h\left( x\right) \cdot u\right\} \,dx$ and they
considered some sharp assumptions on the function $h\left( x\right) $, of
the type $h\in L\left( n,1\right) \left( \Omega ;\mathbb{R}^{m}\right) $ in
dimension $n>2$ (i.e., $\int_{0}^{+\infty }\func{meas}\left\{ x\in \Omega
:\left\vert h\left( x\right) \right\vert >\lambda \right\} ^{1/n}d\lambda
<+\infty $; note that $L^{n+\varepsilon }\subset L\left( n,1\right) \subset
L^{n}$), or $h\in L^{2}\left( \log L\right) ^{\alpha }\left( \Omega ;\mathbb{%
R}^{m}\right) $ for some $\alpha >2$ when $n=2$. Note that these assumptions
on $h$ are independent on the principal part $g\left( \left\vert \xi
\right\vert \right) $. Beck-Mingione obtained the local boundedness of the
gradient $Du$ of the local minimizer under some growth assumptions on $%
g\left( \left\vert \xi \right\vert \right) $, which however is assumed to be
independent of $x$.

In this paper we allow $x-$dependence in the principal part of the
energy-integrand; i.e., under the notation $\left\vert \xi \right\vert =t$,
we consider a general integrand of the form $g=g\left( x,t\right) $, with $%
g:\Omega \times \left[ 0,\infty \right) \rightarrow \left[ 0,\infty \right) $
Carath\'{e}odory function, convex and increasing with respect to $t\in \left[
0,\infty \right) $. Our assumptions, stated below in (\ref{main assumptions}%
), allow us to consider \textit{both fast and slow growth} on the integrand $%
g\left( x,\left\vert Du\right\vert \right) $.

Model energy-integrals that we have in mind are, for instance, \textit{%
exponential growth} with local Lipschitz continuous coefficients $a,b$ ($%
a\left( x\right) ,b\left( x\right) \geq c>0$) 
\begin{equation}
\int_{\Omega }e^{a\left( x\right) \left\vert Du\right\vert
^{2}}\,dx\;\;\;\;or\;\;\;\;\int_{\Omega }b\left( x\right) \exp \left( \ldots
\exp \left( a\left( x\right) \left\vert Du\right\vert ^{2}\right) \right)
\,dx\,;  \label{example 0}
\end{equation}%
\textit{variable exponents} ($a,p\in W_{\limfunc{loc}}^{1,\infty }\left(
\Omega \right) $, $a\left( x\right) \geq c>0$ and $p\left( x\right) \geq p>1$%
) 
\begin{equation}
\int_{\Omega }a\left( x\right) \left\vert Du\right\vert ^{p\left( x\right)
}\,dx\,\;\;\;\;\text{or}\;\;\;\;\int_{\Omega }a\left( x\right) \left(
1+\left\vert Du\right\vert ^{2}\right) ^{p\left( x\right) /2}\,dx\,;
\label{example 2}
\end{equation}%
of course the classical $p-$\textit{Laplacian} energy-integral, with a
constant $p$ strictly greater than $1$ and integrand $f\left( x,Du\right)
=a\left( x\right) \left\vert Du\right\vert ^{p}$, is covered by the example (%
\ref{example 2}): the theory considered here and the Theorem \ref{Theorem 1}
below apply to the $p-$Laplacian. Also \textit{Orlicz-type} energy-integrals
(see Chlebicka \cite{Chlebicka 2018}, Chlebicka \textit{et al.} \cite%
{Chlebicka 2019}), again with local Lipschitz continuous exponent $p\left(
x\right) \geq p>1$, of the type 
\begin{equation}
\int_{\Omega }a\left( x\right) |Du|^{p\left( x\right) }\log (1+|Du|)\,dx\,;
\label{example 2bis}
\end{equation}%
note that the a-priori estimate in Theorem \ref{Theorem 1} below holds also
for some cases with \textit{slow growth}, i.e., when $p\left( x\right) \geq 1
$, in particular when $p\left( x\right) $ is identically equal to $1$. See (%
\ref{example 2 bis again}) and the details in the next section. A class of
energy-integrals of the form 
\begin{equation}
\int_{\Omega }h\left( a\left( x\right) \left\vert Du\right\vert \right)
\,dx\,\;\;\;\;\text{or}\;\;\;\;\int_{\Omega }b\left( x\right) h\left(
a\left( x\right) \left\vert Du\right\vert \right) \,dx\,,  \label{example 3}
\end{equation}%
with $a\left( x\right) ,b\left( x\right) $ locally Lipschitz continuous and
nonnegative coefficients in $\Omega $ and $h:\left[ 0,+\infty \right)
\rightarrow \left[ 0,+\infty \right) $ convex increasing function of class $%
W_{\limfunc{loc}}^{2,\infty }\left( \left[ 0,+\infty \right) \right) $ as in
(\ref{the function h(t)}) below. Also some $g\left( x,\left\vert \xi
\right\vert \right) $ with \textit{slow growth}, precisely \textit{linear
growth} as $t=\left\vert Du\right\vert \rightarrow +\infty $, such as, for $%
n=2,3$, 
\begin{equation}
\int_{\Omega }\left\{ \left\vert Du\right\vert -a\left( x\right) \sqrt{%
\left\vert Du\right\vert }\right\} \,dx\,,  \label{example 4}
\end{equation}%
with $a\in W_{\limfunc{loc}}^{1,\infty }\left( \Omega \right) $, $a\left(
x\right) \geq c>0$ (here more precisely $t\rightarrow t-a\left( x\right) 
\sqrt{t}$ means a smooth convex function in $\left[ 0,+\infty \right) $,
with derivative equal to zero at $t=0$, which coincide with $t-a\left(
x\right) \sqrt{t}$ for $t\geq t_{0}$, for a given $t_{0}>0$, and for $x\in
\Omega $). Some of these examples are covered by the regularity theories
already in the literature; for instance, as already quoted, the paper \cite%
{Uhlenbeck 1977} by Uhlenbeck for the case $g\left( x,\left\vert \xi
\right\vert \right) =\left\vert \xi \right\vert ^{p}$ with exponents $p\geq 2
$, \cite{Marcellini 1996} and \cite{Marcellini-Papi 2006} with general
integrands $g\left( \left\vert \xi \right\vert \right) $ with \textit{%
exponential growth} too, Mascolo-Migliorini \cite{Mascolo-Migliorini 2003}
with integrands $g\left( x,\left\vert \xi \right\vert \right) $ not allowing
slow growth, Beck-Mingione \cite{Beck-Mingione 2019} with $x-$dependence on
the lower order terms.

For completeness related to these researches we mention the \textit{double
phase problems}, recently intensively studied by Colombo-Mingione \cite%
{Colombo-Mingione 2015a}, \cite{Colombo-Mingione 2015b}
Baroni-Colombo-Mingione \cite{Baroni-Colombo-Mingione 2015}, \cite%
{Baroni-Colombo-Mingione 2016}, \cite{Baroni-Colombo-Mingione 2018} and the 
\textit{double phase} with \textit{variable exponents} by
Eleuteri-Marcellini-Mascolo \cite{Eleuteri-Marcellini-Mascolo 2016a}, \cite%
{Eleuteri-Marcellini-Mascolo 2016b}, \cite{Eleuteri-Marcellini-Mascolo 2018}%
. See also Esposito-Leonetti-Mingione \cite{Esposito-Leonetti-Mingione 2004}%
, R\v{a}dulescu-Zhang \cite{Radulescu-Zhang 2018}, Cencelja-R\u{a}%
dulescu-Repov\v{s} \cite{Cencelja-Radulescu-Repovs 2018} and De Filippis 
\cite{DeFilippis 2018}. For related recent references we quote \cite%
{Marcellini 2019}, \cite{Marcellini DiscrContDinSystems 2019} and
Bousquet-Brasco \cite{Bousquet-Brasco 2019},
Carozza-Giannetti-Leonetti-Passarelli \cite%
{Carozza-Giannetti-Leonetti-Passarelli 2018},
Cupini-Giannetti-Giova-Passarelli \cite{Cupini-Giannetti-Giova-Passarelli
2018}, Cupini-Marcellini-Ma\-scolo \cite{Cupini-Marcellini-Mascolo 2009}, 
\cite{Cupini-Marcellini-Mascolo 2012}, \cite{Cupini-Marcellini-Mascolo 2018}%
, De Filippis-Mingione \cite{DeFilippis-Mingione 2019}, Harjulehto-H\"{a}st%
\"{o}-Toivanen \cite{Harjulehto-Hasto-Toivanen 2017}, H\"{a}st\"{o}-Ok \cite%
{Hasto-Ok 2019}, Mingione-Palatucci \cite{Mingione-Palatucci 2019}.

Without loss of generality, by changing $g\left( x,t\right) $ with $g\left(
x,t\right) -g\left( x,0\right) $ if necessary, we can reduce ourselves to
the case $g(x,0)=0$ for almost every $x\in \Omega $. We assume that the
partial derivatives $g_{t}$, $g_{tt}$, $g_{tx_{k}}$ exist (for every $%
k=1,2,\ldots n$) and that they are Carath\'{e}odory functions too, with $%
g_{t}\left( x,0\right) =0$.

In the next section we show that the following assumptions (\ref{main
assumptions}), (\ref{the function h(t)}) cover the model examples from (\ref%
{example 0}) to (\ref{example 4}). Precisely, we require that the following
growth conditions hold: let $t_{0}>0$ be fixed; for every open subset $%
\Omega ^{\prime }$ compactly contained in $\Omega $, there exist $\vartheta
\geq 1$ and positive constants $m$ and $M_{\vartheta }$ such that 
\begin{equation}
\left\{ 
\begin{array}{l}
mh^{\prime }\left( t\right) \leq g_{t}\left( x,t\right) \leq M_{\vartheta }
\left[ h^{\prime }\left( t\right) \right] ^{\vartheta }t^{1-\vartheta } \\ 
mh^{\prime \prime }\left( t\right) \leq g_{tt}\left( x,t\right) \leq
M_{\vartheta }\left[ h^{\prime \prime }\left( t\right) \right] ^{\vartheta }
\\ 
\left\vert g_{tx_{k}}\left( x,t\right) \right\vert \leq M_{\vartheta }\min
\left\{ g_{t}(x,t),t\,g_{tt}\left( x,t\right) \right\} ^{\vartheta }%
\end{array}%
\right.   \label{main assumptions}
\end{equation}%
for every $t\geq t_{0}$ and for $x\in \Omega ^{\prime }$. The role of the
parameter $\vartheta $ can be easily understood if we compare (\ref{main
assumptions}) with the above model examples; see the details in the next
section. The special case $\vartheta =1$ corresponds to the so called 
\textit{natural growth conditions}. Here, following similar assumptions in 
\cite{Marcellini-Papi 2006}, $h:\left[ 0,+\infty \right) \rightarrow \left[
0,+\infty \right) $ is a convex increasing function of class $W_{\limfunc{loc%
}}^{2,\infty }$ satisfying the following property: for some $\beta $, $\frac{%
1}{n}<\beta <\frac{2}{n}$, $(2\vartheta -1)\vartheta <\left( 1-\beta
\right) \frac{2^{\ast}}{2}$, and for every $\alpha $ such that $1<\alpha \leq \frac{n}{%
n-1}$, there exist constants $m_{\beta }$ and $M_{\alpha }$ such that 
\begin{equation}
\frac{m_{\beta }}{t^{2\beta }}\left[ \left( \frac{h^{\prime }\left( t\right) 
}{t}\right) ^{\frac{n-2}{n}}+\frac{h^{\prime }\left( t\right) }{t}\right]
\leq h^{\prime \prime }\left( t\right) \leq M_{\alpha }\left[ \left( \frac{%
h^{\prime }\left( t\right) }{t}\right) ^{\alpha }+\frac{h^{\prime }\left(
t\right) }{t}\right]   \label{the function h(t)}
\end{equation}%
for every $t\geq t_{0}$. We obtain the following a-priori gradient estimate.

\begin{theorem}
\label{Theorem 1} Let us assume that conditions (\ref{main assumptions}),(%
\ref{the function h(t)}) hold. Then the gradient of any smooth local
minimizer of the integral (\ref{energy-integral 1}) is uniformly locally
bounded in $\Omega $. Precisely, there exists $\varepsilon >0$ and a
positive constant $C$ such that,%
\begin{equation}
\left\Vert Du\right\Vert _{L^{\infty }\left( B_{\rho },\mathbb{R}^{m\times
n}\right) }\leq C\left\{ \int_{B_{R}}\left( 1+g\left( x,\left\vert
Du\right\vert \right) \right) \,dx\right\} ^{1+\varepsilon },
\label{conclusion in the main theorem}
\end{equation}%
for every $\rho ,R$ $(0<\rho <R)$. The constant $C$ depends on $%
n,\varepsilon ,\vartheta ,\rho ,R,t_{0},\beta,\alpha$ and $\sup \left\{
h^{\prime \prime }(t):\;t\in \left[ 0,t_{0}\right] \right\} $, while $%
\varepsilon $ depends on $\vartheta ,\beta ,n$.
\end{theorem}

As described above, Theorem \ref{Theorem 1} gives an \textit{a-priori local
gradient bound}. In the regularity theory for weak solutions this is the
main step to get the \textit{local Lipschitz continuity} of solutions, since
the minimizer is assumed to be smooth enough for the validity of the Euler's
first and second variation (see also the statement of Theorem \ref{thm_1}),
but the constants in the bound (\ref{conclusion in the main theorem}) do not
depend on this smoothness. An approximation argument gives this \textit{%
local Lipschitz continuity} property. In fact, by applying the a-priori
gradient estimate to an approximating energy integrand $f_{k}\left(
x,\left\vert \xi \right\vert \right) $ which converges to $f\left(
x,\left\vert \xi \right\vert \right) $ as $k\rightarrow +\infty $ and which
satisfies standard growth conditions, we obtain a sequence of smooth
approximating solutions $u_{k}$ with 
\begin{equation*}
\left\Vert Du_{k}\right\Vert _{L^{\infty }\left( B_{\rho },\mathbb{R}%
^{m\times n}\right) }\leq const\,,
\end{equation*}%
for every fixed small radius $\rho $, and the constant is independent of $k$%
. In the limit as $k\rightarrow +\infty $ also the solution $u$ to the
original variational problem, related to the energy integrand $f\left(
x,\left\vert \xi \right\vert \right) $, comes out to have locally bounded
gradient and thus it is local Lipschitz continuous in $\Omega $. This
approximation procedure is described in details in Section 5 of \cite%
{Marcellini 1996} and in Section 6 of \cite{Marcellini-Papi 2006}; see also
Section 5 in \cite{Beck-Mingione 2019}. The proof of the a-priori gradient
bound (\ref{conclusion in the main theorem}) is already long enough and we
prefer to include the details of the approximating procedure, under the
specific assumptions considered here, in a future article.

In the next Section we describe some examples, in particular the model
examples from (\ref{example 0}) to (\ref{example 4}), and we show the role
of the parameters $\vartheta $, $\beta $, $\alpha $ in the assumptions (\ref%
{main assumptions}), (\ref{the function h(t)}). Then, after some preliminary
results proposed in Section \ref{Section: preliminary results}, in Section %
\ref{Section: A priori estimates} we give the \textit{a-priori estimate} as
in (\ref{conclusion in the main theorem}) and we conclude the proof of
Theorem \ref{Theorem 1}.

\section{Some model examples\label{Section: Some model examples}}

We show in this section that all the above listed model examples satisfy the
assumptions (\ref{main assumptions}),(\ref{the function h(t)}), which -
although technical - however can be considered general enough to cover the
cases from (\ref{example 0}) to (\ref{example 4}) and ensure the gradient
bound in (\ref{conclusion in the main theorem}).

We first note that the parameter $\beta$ has not a relevant role when the variational problem has \textit{fast growth}, i.e. if $\frac{h'(t)}{t}\to\infty$ when $t\to\infty$. In this case, for instance, we can choose $\beta=\frac{3}{2n}$, the intermediate point of the interval $\left( \frac{1}{n}, \frac{2}{n}\right)$, and for $\vartheta$ any real number greater than or equal to $1$ (in some cases $\vartheta$ strictly greater than $1$, see the details below) such that $(2\vartheta-1)\vartheta<\frac{n-\frac{3}{2}}{n-2}$.

We start with the example (\ref{example 0}), with $g\left( x,t\right)
=e^{a\left( x\right) t^{2}}$ and the positive coefficient $a\left( x\right) $
locally Lipschitz continuous in $\Omega $. In order to prove the local $%
L^{\infty }$ bound of the gradient of a local minimizer we first fix a ball $%
B$ compactly contained in $\Omega $ such that the oscillation of $a\left(
x\right) $ is small in $\overline{B}$; precisely, under the notation 
\begin{equation}
a_{M}=\max \left\{ a\left( x\right) :\;x\in \overline{B}\right\}
,\;\;\;\;a_{m}=\min \left\{ a\left( x\right) :\;x\in \overline{B}\right\} ,
\label{am}
\end{equation}%
given $\vartheta >1$ we choose the radius of the ball $B$ small enough such
that $a_{M}\leq \vartheta a_{m}$. Then, if we define $h\left( t\right)
=:e^{a_{m}t^{2}}$, we have 
\begin{equation*}
h(t)=e^{a_{m}t^{2}}\leq e^{a\left( x\right) t^{2}}=g\left( x,t\right) \leq
e^{a_{M}t^{2}}\leq e^{\vartheta a_{m}t^{2}}=\left[ h(t)\right] ^{\vartheta }
\end{equation*}%
for every $t\geq 0$ and $x\in \overline{B}$. Similarly for the derivatives $%
g_{t}\left( x,t\right) =2a\left( x\right) te^{a\left( x\right) t^{2}}$\ and $%
h^{\prime }\left( t\right) =2a_{m}te^{a_{m}t^{2}}$ we obtain 
\begin{equation*}
\frac{h^{\prime }(t)}{t}\leq \frac{g_{t}(x,t)}{t}\leq
2a_{M}e^{a_{M}t^{2}}\leq 2\vartheta a_{m}e^{\vartheta a_{m}t^{2}}
\end{equation*}%
\begin{equation*}
=\left( 2a_{m}\right) ^{\vartheta }\left( 2a_{m}\right) ^{1-\vartheta
}\vartheta \left( e^{a_{m}t^{2}}\right) ^{\vartheta }=\vartheta \left(
2a_{m}\right) ^{1-\vartheta }\left[ \frac{h^{\prime }(t)}{t}\right]
^{\vartheta }
\end{equation*}%
for every $t>0$ and $x\in \overline{B}$. For the second derivatives with
respect to $t$ we have similar estimates: 
\begin{equation*}
h^{\prime \prime }(t)\leq g_{tt}(x,t)\leq 2a_{m}\vartheta e^{\vartheta
a_{m}t^{2}}\left( 1+2\vartheta a_{m}t^{2}\right)
\end{equation*}%
\begin{equation*}
\leq \left( 2a_{m}\right) ^{\vartheta }\left( 2a_{m}\right) ^{1-\vartheta
}\vartheta ^{2}\left[ e^{a_{m}t^{2}}\left( 1+2a_{m}t^{2}\right) \right]
^{\vartheta }=\left( 2a_{m}\right) ^{1-\vartheta }\vartheta ^{2}\left[
h^{\prime \prime }(t)\right] ^{\vartheta }.
\end{equation*}%
Then, if we call $M_{\vartheta }=\max \left\{ \vartheta \left( 2a_{m}\right)
^{1-\vartheta },\vartheta ^{2}\left( 2a_{m}\right) ^{1-\vartheta }\right\} $%
, the first two conditions of (\ref{main assumptions}) hold for any $%
\vartheta>1$.

From $g_{t}\left( x,t\right) =2a\left( x\right) te^{a\left( x\right) t^{2}}$%
, we estimate the mixed derivative $g_{tx_{k}}$. In this case 
\begin{equation*}
\min \left\{ g_{t}(x,t),t\,g_{tt}\left( x,t\right) \right\} =g_{t}(x,t)
\end{equation*}%
and then we consider the quotient 
\begin{equation*}
\frac{\left\vert g_{tx_{k}}(x,t)\right\vert }{g_{t}^{\vartheta }\left(
x,t\right) }=\frac{\left\vert 2a_{x_{k}}\left( x\right) te^{a\left( x\right)
t^{2}}\left( 1+a\left( x\right) t^{2}\right) \right\vert }{\left(
2a(x)te^{a(x)t^{2}}\right) ^{\vartheta }}\,;
\end{equation*}%
if we denote by $L$ the Lipschitz constant of the coefficient $a\left(
x\right) $ in $B$, we have 
\begin{equation*}
\frac{\left\vert g_{tx_{k}}(x,t)\right\vert }{g_{t}^{\vartheta }\left(
x,t\right) }\leq \frac{2L}{(2a_{m})^{\vartheta }}\frac{1+a_{M}t^{2}}{\left(
te^{a(x)t^{2}}\right) ^{\vartheta -1}}
\end{equation*}%
and, for every $\vartheta >1$, the right hand side is bounded in $\overline{B%
}$ for every $t>1$. Note the crucial role of the parameter $\vartheta $
strictly greater than $1$. Therefore also the last condition in (\ref{main
assumptions}) is satisfied.

It remains to verify the condition (\ref{the function h(t)}) for the
function $h\left( t\right) =e^{a_{m}t^{2}}$. Here the parameter $\alpha >0$
plays a crucial role, since $h^{\prime }(t)=2a_{m}te^{a_{m}t^{2}}$ and $%
h^{\prime \prime }(t)=2a_{m}e^{a_{m}t^{2}}+\left( 2a_{m}t\right)
^{2}e^{a_{m}t^{2}}$ and we cannot bound $h^{\prime \prime }(t)$ in terms of $%
\frac{h^{\prime }(t)}{t}$. On the contrary, for every $\alpha >1$ there
exists a constant $M_{\alpha }$ such that the following bound holds, which
implies the bound required in (\ref{the function h(t)}), 
\begin{equation*}
h^{\prime \prime }\left( t\right) \leq M_{\alpha }\left( \frac{h^{\prime
}\left( t\right) }{t}\right) ^{\alpha }\,,\;\;\;\;\;\forall \;t>0.
\end{equation*}%
The left hand side inequality in (\ref{the function h(t)}) is satisfied,
since in this case $h^{\prime \prime }\left( t\right) \geq \frac{h^{\prime
}\left( t\right) }{t}$ for every $t>0$ and the quantity $\left( \frac{%
h^{\prime }\left( t\right) }{t}\right) ^{\frac{n-2}{n}}$ goes to $+\infty $
slower than $\frac{h^{\prime }\left( t\right) }{t}$. Moreover $t^{-2\beta
}\rightarrow 0$ as $t\rightarrow +\infty $. 

The example (\ref{example 0}), by changing $a\left( x\right) $ with $%
a^{2}\left( x\right) $, can be equivalently written in the form 
\begin{equation*}
\int_{\Omega }e^{a^{2}\left( x\right) \left\vert Du\right\vert
^{2}}\,dx=\int_{\Omega }e^{\left( a\left( x\right) \left\vert Du\right\vert
\right) ^{2}}\,dx
\end{equation*}%
and can be considered an example in the class of energy-integrals as in (\ref%
{example 3}), of the type 
\begin{equation}
\int_{\Omega }h\left( a\left( x\right) \left\vert Du\right\vert \right) \,dx
\label{energy-integral with h}
\end{equation}%
with $h:\left[ 0,+\infty \right) \rightarrow \left[ 0,+\infty \right) $
convex increasing function of class $W_{\limfunc{loc}}^{2,\infty }\left( %
\left[ 0,+\infty \right) \right) $ satisfying (\ref{the function h(t)}). If $%
a\left( x\right) $ is a positive locally Lipschitz continuous coefficient in 
$\Omega $ we can use a similar argument as above and obtain, by Theorem \ref%
{Theorem 1}, the a-priori estimate also of minima of the integral (\ref%
{energy-integral with h}).

The example (\ref{example 2}) is similar to (\ref{example 0}). In this case
we have to test the conditions in (\ref{main assumptions}). Under the
notation $g\left( x,t\right) =t^{p\left( x\right) }$ (for simplicity we
consider here $a\left( x\right) $ identically equal to $1$), we have $%
g_{t}\left( x,t\right) =p\left( x\right) t^{p\left( x\right) -1}$ and 
\begin{equation*}
g_{tx_{k}}\left( x,t\right) =p_{x_{k}}\left( x\right) t^{p\left( x\right)
-1}+p\left( x\right) \frac{\partial }{\partial x_{k}}\left[ e^{\left(
p\left( x\right) -1\right) \log t}\right] 
\end{equation*}%
\begin{equation*}
=p_{x_{k}}\left( x\right) t^{p\left( x\right) -1}\left[ 1+p\left( x\right)
\log t\right] \,.
\end{equation*}%
If we denote by $L$ the Lipschitz constant of $p\left( x\right) $ on a fixed
open subset $\Omega ^{\prime }$ whose closure is contained in $\Omega $,
then 
\begin{equation*}
\frac{\left\vert g_{tx_{k}}(x,t)\right\vert }{\left( g_{tt}\left( x,t\right)
t\right) ^{\vartheta }}\leq L\frac{1+p\left( x\right) \log t}{p^{\vartheta
}\left( x\right) \left( p\left( x\right) -1\right) ^{\vartheta
}t^{(\vartheta -1)(p(x)-1)}}
\end{equation*}%
and thus the quotient is bounded for $t\in \left[ 1,+\infty \right) $ and $%
x\in \Omega ^{\prime }$ if $p\left( x\right) >1$ is locally Lipschitz
continuous in $\Omega $ (i.e., also being $p\left( x\right) \geq c>1$ for
some constant $c=c\left( \Omega ^{\prime }\right) $) and $\vartheta >1$.
Also here note the role of the parameter $\vartheta $ strictly greater than $%
1.$ Since $g_{tt}\left( x,t\right) t$ and $g_{t}\left( x,t\right) $ are of
the same order as $t\rightarrow +\infty $, similarly 
\begin{equation*}
\frac{\left\vert g_{tx_{k}}(x,t)\right\vert }{g_{t}^{\vartheta }\left(
x,t\right) }\leq L\frac{1+p(x)\log (t)}{p^{\vartheta }(x)t^{(\vartheta
-1)(p(x)-1)}}.
\end{equation*}%
The other conditions in (\ref{main assumptions}) can be tested as before.

Similar computations can be carried out for the example (\ref{example 2bis}%
), with $g\left( x,t\right) =a\left( x\right) t^{p\left( x\right) }\log
(1+t) $, under the assumption that the coefficient $a\left( x\right) $ and
the exponent $p\left( x\right) $ are locally Lipschitz continuous in $\Omega 
$ and that, for every $\Omega ^{\prime }$ compactly contained in $\Omega $,
there exists a constant $c>1$ such that $p\left( x\right) \geq c$ for every $%
x\in \Omega ^{\prime }$. Also the limit case enters in this regularity
theory, when the exponent $p\left( x\right) \geq 1$ for every $x\in \Omega $%
, however by assuming in this special case that $a\left( x\right) $ is
identically equal to $1$; see the details below in (\ref{example 2 bis again}%
).

We now consider the \textit{slow growth} example (\ref{example 4}). Here 
\begin{equation}
g\left( x,t\right) =t-a\left( x\right) \sqrt{t}\,.  \label{square root}
\end{equation}%
As already mentioned, $g\left( x,t\right) $ in (\ref{square root}) means a
smooth convex function in $\left[ 0,+\infty \right) $, with derivative equal
to zero at $t=0$, which coincide with $t-a\left( x\right) \sqrt{t}$ for $%
t\geq t_{0}$, for a given $t_{0}>0$ and for $x\in \Omega $. Again, we use
the notation in (\ref{am}), precisely given $\Omega ^{\prime }\subset
\subset \Omega $, 
\begin{equation*}
a_{M}=\max \left\{ a\left( x\right) :\;x\in \Omega ^{\prime }\right\}
,\;\;\;\;a_{m}=\min \left\{ a\left( x\right) :\;x\in \Omega ^{\prime
}\right\} ,
\end{equation*}%
and we require $a_{m}$ to be positive. Then, for $x\in \Omega ^{\prime }$
and $t\geq t_{0}$, 
\begin{equation*}
\left\{ 
\begin{array}{l}
g_{t}\left( x,t\right) =1-\frac{1}{2}a\left( x\right) t^{-\frac{1}{2}}\geq 1-%
\frac{1}{2}a_{M}t^{-\frac{1}{2}} \\ 
t\,g_{tt}\left( x,t\right) =\frac{1}{4}a\left( x\right) t^{-\frac{1}{2}}\leq 
\frac{1}{4}a_{M}t^{-\frac{1}{2}}%
\end{array}%
\right.
\end{equation*}%
and for large $t$ we have 
\begin{equation*}
\min \left\{ g_{t}(x,t),t\,g_{tt}\left( x,t\right) \right\} =t\,g_{tt}\left(
x,t\right) .
\end{equation*}%
If we denote by $L$ the Lipschitz constants of $a\left( x\right) $ on $%
\Omega ^{\prime }\subset \subset \Omega $, we fix $\vartheta =1$ and we
obtain the bounded quotient for $t\geq t_{0}$%
\begin{equation*}
\frac{\left\vert g_{tx_{k}}(x,t)\right\vert }{\min \left\{
g_{t}(x,t),t\,g_{tt}\left( x,t\right) \right\} }=\frac{\left\vert
g_{tx_{k}}(x,t)\right\vert }{t\,g_{tt}\left( x,t\right) }\leq \frac{\frac{L}{%
2}t^{-\frac{1}{2}}}{\frac{1}{4}a_{m}t^{-\frac{1}{2}}}=\frac{2L}{a_{m}}.
\end{equation*}%
In order to test the other conditions in (\ref{main assumptions}), we define 
$h\left( t\right) =:t-\sqrt{t}$ and, for $\left( x,t\right) \in \Omega
^{\prime }\times \left( 0,+\infty \right) $, we have 
\begin{equation*}
\min \left\{ 1;a_{M}\right\} h^{\prime }\left( t\right) \leq g_{t}\left(
x,t\right) \leq \max \left\{ 1;a_{m}\right\} h^{\prime }\left( t\right) \,,
\end{equation*}%
\begin{equation*}
a_{m}h^{\prime \prime }\left( t\right) \leq g_{tt}\left( x,t\right) \leq
a_{M}h^{\prime \prime }\left( t\right) \,.
\end{equation*}%
Then (\ref{main assumptions}) are satisfied with $\vartheta =1$. Finally the
convex function $h\left( t\right) =:t-\sqrt{t}$ satisfies (\ref{the function
h(t)}). In fact, since $h^{\prime \prime }\left( t\right) =\frac{1}{4}t^{-%
\frac{3}{2}}$ and $\frac{h^{\prime }\left( t\right) }{t}=t^{-1}-\frac{1}{2}%
t^{-\frac{3}{2}}$, then as $t\rightarrow +\infty $, $h^{\prime \prime
}\left( t\right) $ goes to zero faster than $\frac{h^{\prime }\left(
t\right) }{t}$ and for every $\alpha >1$ there exists a constant $M_{\alpha
} $ such that 
\begin{equation*}
h^{\prime \prime }\left( t\right) \leq M_{\alpha }\left[ \left( \frac{%
h^{\prime }\left( t\right) }{t}\right) ^{\alpha }+\frac{h^{\prime }\left(
t\right) }{t}\right] \,,\;\;\;\;\;\forall \;t\geq 1.
\end{equation*}%
On the other side, since as $t\rightarrow +\infty $ the quantity $\frac{%
h^{\prime }\left( t\right) }{t}\rightarrow 0$ and $\frac{n-2}{n}<1$, then $%
\frac{h^{\prime }\left( t\right) }{t}$ goes to zero faster than $\left( 
\frac{h^{\prime }\left( t\right) }{t}\right) ^{\frac{n-2}{n}}$. Therefore we
can equivalently test the condition 
\begin{equation}
\frac{m_{\beta }}{t^{2\beta }}\left( \frac{h^{\prime }\left( t\right) }{t}%
\right) ^{\frac{n-2}{n}}\leq h^{\prime \prime }\left( t\right)
\,,\;\;\;\;\;\forall \;t\geq 1.  \label{condition on h in the linear case}
\end{equation}%
The order of infinitesimal of the left hand side in (\ref{condition on h in
the linear case}) is $\frac{1}{t^{2\beta +\frac{n-2}{n}}}$, while the order
of infinitesimal of the right hand side is $\frac{1}{t^{\frac{3}{2}}}$.
Therefore condition (\ref{condition on h in the linear case}) is satisfied
for some constant $m_{\beta }$ if 
\begin{equation*}
2\beta +\frac{n-2}{n}\geq \frac{3}{2}.
\end{equation*}%
This, together with the condition $\frac{1}{n}<\beta <\frac{2}{n}$, gives
the condition for $\beta $ 
\begin{equation*}
\frac{1}{4}+\frac{1}{n}\leq \beta <\frac{2}{n}\,,
\end{equation*}%
which is compatible if $n=2,3$.

With a similar computation we can treat the example (\ref{example 2bis}) for
general locally Lipschitz continuous coefficients $a\left( x\right) $ and
exponents $p\left( x\right) $, by assuming that there exists a constant $c>1$
such that $p\left( x\right) \geq c$ for every $x\in \Omega $. While, under
the more general assumption $p\left( x\right) \geq 1$ for every $x\in \Omega 
$, we need to assume $a\left( x\right) $ identically equal to $1$. I.e., for
instance, if $p\left( x\right) $ is identically equal to $1$, then we
consider the energy integral 
\begin{equation}
\int_{\Omega }|Du|\log (1+|Du|)\,dx\,;  \label{example 2 bis again}
\end{equation}%
here the function $h$ is defined by 
\begin{equation*}
h\left( t\right) =t\log \left( 1+t\right) 
\end{equation*}%
and comes out that the right inequality in (\ref{the function h(t)}) is
satisfied for some constant $M$: 
\begin{equation*}
h^{\prime \prime }\left( t\right) =\frac{2+t}{\left( 1+t\right) ^{2}}\leq M%
\frac{h^{\prime }\left( t\right) }{t}=M\left( \frac{\log \left( 1+t\right) }{%
t}+\frac{1}{1+t}\right) ,\;\;\;\;\;\forall \;t\geq 1.
\end{equation*}%
While in the left hand side of (\ref{the function h(t)}), since as $%
t\rightarrow +\infty $ the quantity $\frac{h^{\prime }\left( t\right) }{t}$
converges to $0$, we test (\ref{condition on h in the linear case}) and the
problem is to compare the order of infinitesimal of the left hand side in (%
\ref{condition on h in the linear case}), which is $\frac{\left[ \log \left(
1+t\right) \right] ^{\frac{n-2}{n}}}{t^{2\beta +\frac{n-2}{n}}}$, with the
order of infinitesimal of the right hand side, equal to $\frac{1}{t}$. A
sufficient condition in this case is $2\beta +\frac{n-2}{n}>1$ (with the
strict sign inequality), which is compatible with $\frac{1}{n}<\beta <\frac{2%
}{n}$ for every $n\in \mathbb{N}$, $n\geq 2$.

\section{Some preliminary lemmata\label{Section: preliminary results}}

In the following we consider test maps $\varphi =\left( \varphi ^{\alpha
}\right) _{\alpha =1,2,\ldots ,n}$ with components of the form $\varphi
^{\alpha }=\eta ^{2}u_{x_{k}}^{\alpha }\Phi \left( \left\vert Du\right\vert
\right) $, where $\eta \in C_{0}^{1}\left( \Omega \right) $ and $\Phi =\Phi
\left( t\right) $ is a real nonnegative function, defined for $t\in \left[
0,+\infty \right) $. We consider $\Phi \left( t\right) $ depending on a real
parameter $\gamma \geq 0$ (in general without denoting explicitly this
dependence) by separating two cases: the first one with $\gamma $ large and
the second one when $\gamma $ is small. In order to simplify the proofs,
throughout the paper we will assume that $t_0=1$ and $g(x, 1)>0$ for almost
every $x\in \Omega$. Precisely, if $\gamma >1$ we define 
\begin{equation}  \label{Phi}
\Phi \left( t\right)= 
\begin{cases}
0 & \text{if $0\leq t\leq 1$,} \\ 
\left( t-1\right) ^{\gamma } & \text{if $t>1$}.%
\end{cases}%
\end{equation}
The following simple inequality holds.

\begin{lemma}
\label{Lemma for Phi}For every $\gamma \in \left( 1,+\infty \right) $ the
function $\Phi \left( t\right) $ defined in (\ref{Phi}) satisfies the
inequality 
\begin{equation}
0\leq \Phi ^{\prime }\left( t\right) t\leq \gamma \left( 1+2\Phi \left(
t\right) \right) \,,\;\;\;\forall \;t\geq 0.  \label{inequality for Phi}
\end{equation}
\end{lemma}

\begin{proof}
The inequality (\ref{inequality for Phi}) is trivial when $t\in \left[ 0,1%
\right] $. If $t>1$ then 
\begin{equation}
\Phi ^{\prime }\left( t\right) t=\gamma \left( t-1\right) ^{\gamma
-1}t=\gamma \left( t-1\right) ^{\gamma }+\gamma \left( t-1\right) ^{\gamma
-1}=\gamma \Phi \left( t\right) +\gamma \left( t-1\right) ^{\gamma -1}.
\label{inequality for Phi 1}
\end{equation}%
Since $\left( t-1\right) ^{\gamma -1}\leq \left( t-1\right) ^{\gamma }$ when 
$t\geq 2$, while $\left( t-1\right) ^{\gamma -1}\leq 1$ if $t\in \left[ 1,2%
\right] $, then in any case 
\begin{equation}
\left( t-1\right) ^{\gamma -1}\leq 1+\left( t-1\right) ^{\gamma }=1+\Phi
\left( t\right) ,\;\;\;\;\forall \;t\geq 1.  \label{inequality for Phi 2}
\end{equation}%
From (\ref{inequality for Phi 1}),(\ref{inequality for Phi 2}) we get the
conclusion (\ref{inequality for Phi}).
\end{proof}

For $\gamma \in \left[ 0,1\right] $ we define 
\begin{equation}  \label{Phi - the case beta small}
\Phi \left( t\right) = 
\begin{cases}
0 & \text{if $0\leq t\leq 1$}, \\ 
\left( t-1\right) ^{2}t^{\gamma -2} & \text{if $t>1$}.%
\end{cases}%
\end{equation}

\begin{lemma}
\label{Lemma for Phi - the case beta small}For every $\gamma \in \left[ 0,1%
\right] $ the function $\Phi \left( t\right) $ defined in (\ref{Phi - the
case beta small}) satisfies the inequality 
\begin{equation}
0\leq \Phi ^{\prime }\left( t\right) t\leq 2+\left( \gamma +2\right) \Phi
\left( t\right) \,,\;\;\;\forall \;t\geq 0.
\label{inequality for Phi - the case beta small}
\end{equation}
\end{lemma}

\begin{proof}
The inequality (\ref{inequality for Phi}) is satisfied when $t\in \left[ 0,1%
\right] $. If $t>1$ then 
\begin{equation*}
\Phi ^{\prime }\left( t\right) =2\left( t-1\right) t^{\gamma -2}+\left(
\gamma -2\right) \left( t-1\right) ^{2}t^{\gamma -3}
\end{equation*}%
\begin{equation*}
=\left( t-1\right) t^{\gamma -3}\left[ 2t+\left( \gamma -2\right) \left(
t-1\right) \right] =\left( t-1\right) t^{\gamma -3}\left[ \gamma \left(
t-1\right) +2\right]
\end{equation*}%
and thus $\Phi ^{\prime }\left( t\right) \geq 0$ and 
\begin{equation}
\Phi ^{\prime }\left( t\right) t=\gamma \left( t-1\right) ^{2}t^{\gamma
-2}+2\left( t-1\right) t^{\gamma -2}=\gamma \Phi \left( t\right) +2\left(
t-1\right) t^{\gamma -2}.  \label{inequality for Phi 3}
\end{equation}%
Since $\left( t-1\right) t^{\gamma -2}\leq \left( t-1\right) ^{2}t^{\gamma
-2}$ when $t\geq 2$, while if $t\in \left[ 1,2\right] $ then $\left(
t-1\right) t^{\gamma -2}\leq t^{\gamma -2}\leq 1$ since $\gamma -2\leq 0$;
therefore in any case 
\begin{equation}
\left( t-1\right) t^{\gamma -2}\leq \left( t-1\right) ^{2}t^{\gamma
-2}+1=1+\Phi \left( t\right) ,\;\;\;\;\forall \;t\geq 1.
\label{inequality for Phi 4}
\end{equation}%
From (\ref{inequality for Phi 3}),(\ref{inequality for Phi 4}) we get the
thesis (\ref{inequality for Phi - the case beta small}).
\end{proof}

\begin{remark}
In the next section we consider real nonnegative functions $\Phi =\Phi
\left( t\right) $ as in (\ref{Phi}) when $\gamma \in \left( 1,+\infty
\right) $ or as in (\ref{Phi - the case beta small}) when $\gamma \in \left[
0, 1\right] $. As consequence of Lemmata \ref{Lemma for Phi} and \ref{Lemma
for Phi - the case beta small}, for every $\gamma \in \left[ 0,+\infty
\right) $ we are allowed to consider functions $\Phi _{\gamma }:\left[
0,+\infty \right) \rightarrow \left[ 0,+\infty \right) $ (later we do not
denote explicitly the dependence on the parameter $\gamma $), which are
increasing and of class $C^{1}$ in $\left[ 0,+\infty \right) $, identically
equal to zero when $t\in \left[ 0,1\right] $ and satisfy the growth
conditions 
\begin{equation}
\left\{ 
\begin{array}{l}
0\leq \Phi _{\gamma }\left( t\right) \leq t^{\gamma } \\ 
0\leq \Phi _{\gamma }^{\prime }\left( t\right) t\leq \max \left\{ 2;\gamma
\right\} +\max \left\{ 2\gamma ;\gamma +2\right\} \Phi _{\gamma }\left(
t\right)%
\end{array}%
\right. ,\;\;\;\forall \;t\geq 0,\;\;\forall \;\gamma \geq 0,
\label{inequality for Phi - both cases}
\end{equation}%
and the second inequality is implied and can be also simply written, for
instance, in the form 
\begin{equation}
0\leq \Phi _{\gamma }^{\prime }\left( t\right) t\leq \left( 2\gamma
+2\right) \left( 1+\Phi _{\gamma }\left( t\right) \right) ,\;\;\;\forall
\;t\geq 0,\;\;\forall \;\gamma \geq 0.  \label{inequality for Phi - final}
\end{equation}
\end{remark}

In what follows, we will also use the functions 
\begin{equation}  \label{def_K}
\mathcal{K}_M(t)=\max\left\{ h^{\prime\prime}(t), \frac{h^{\prime}(t)}{t}%
\right\}
\end{equation}
and 
\begin{equation}  \label{def_K_min}
\mathcal{K}_m(t)=\min\left\{ h^{\prime\prime}(t), \frac{h^{\prime}(t)}{t}%
\right\}
\end{equation}
related to the function $h$ defined in (\ref{the function h(t)}). We shall
use the following lemma when $\gamma\geq 1$.

\begin{lemma}
\label{growth_lemma} Let $h$ satisfy (\ref{the function h(t)}) and let $%
\mathcal{K}_m, \mathcal{K}_M$ be the functions defined in (\ref{def_K}), (%
\ref{def_K_min}). Then, for every $\sigma$ with $\frac{2\alpha }{%
2^{\ast}(2-\alpha) }\leq \sigma \leq 1$ and for every $\gamma \geq 1$ there
exists a constant $C$ (depending on $\alpha$) such that

\begin{equation}  \label{sec_der_lemma_2}
1+\int_{1}^{t}(s-1)^{\gamma }\sqrt{\mathcal{K}_m(s)}\,ds\geq C\left[
1+\left( \frac{(t-1)^{\gamma+1}t^{-\beta}}{\gamma+1}\right)^{2^{\ast}} 
\mathcal{K}_M^{\frac{1}{\sigma}}(t) \right]^{\frac{1}{2^{\ast}}},
\end{equation}
for every $t\geq 1$.
\end{lemma}

\begin{proof}
Let us define $\delta =2^{\ast }\sigma $, then we observe that 
\begin{equation*}
\left[ 1+\left( \frac{(t-1)^{\gamma +1}t^{-\beta }}{\gamma +1}\right)
^{2^{\ast }}\mathcal{K}_M^{\frac{1}{\sigma }}(t)\right] ^{\frac{1}{2^{\ast }}%
}\leq \left[ 1+\frac{(t-1)^{\gamma +1}t^{-\beta }}{\gamma +1}\mathcal{K}_M^{%
\frac{1}{\delta }}(t)\right] ,
\end{equation*}%
for every $t\geq 1$ and for every $\gamma \geq 1$. By definition of $%
\mathcal{K}$ we get 
\begin{equation*}
\mathcal{K}_M(t)\leq \frac{h^{\prime }(t)}{t}+h^{\prime \prime }(t)
\end{equation*}%
and, by the right hand side of (\ref{the function h(t)}) we can write 
\begin{equation*}
\mathcal{K}_M(t)\leq (m_{\alpha }+1)\left[ \frac{h^{\prime }(t)}{t}+\left( 
\frac{h^{\prime }(t)}{t}\right) ^{\alpha }\right] .
\end{equation*}%
From these, instead of proving (\ref{sec_der_lemma_2}) we can prove that 
\begin{equation}
1+\int_{1}^{t}(s-1)^{\gamma }\sqrt{\mathcal{K}_m(s)}\,ds  \label{MP_lemma_1}
\end{equation}%
\begin{equation*}
\geq C\left[ 1+\frac{(t-1)^{\gamma +1}t^{-\beta }}{\gamma +1}\left( \frac{%
h^{\prime }(t)}{t}+\left( \frac{h^{\prime }(t)}{t}\right) ^{\alpha }\right)
^{\frac{1}{\delta }}\right] \frac{{}}{{}}
\end{equation*}%
for every $t\geq 1$. At this end it is sufficient to show the inequality
between the derivatives side to side with respect to $t$ of (\ref{MP_lemma_1}%
), i.e., since $\frac{\gamma +1-\frac{t-1}{t}\beta}{\gamma +1}$, $\frac{1}{%
\gamma +1}$, $\frac{t-1}{t}<1$, 
\begin{equation*}
\sqrt{\mathcal{K}_m(t)}\geq Ct^{-\beta }\left[ \left( \frac{h^{\prime }(t)}{t%
}\right) ^{\frac{1}{\delta }}+\left( \frac{h^{\prime }(t)}{t}\right) ^{\frac{%
\alpha }{\delta }}+\Biggl( \frac{h^{\prime }(t)}{t}\right) ^{\frac{1}{\delta 
}-1}h^{\prime \prime }(t)
\end{equation*}
\begin{equation*}
+\left( \frac{h^{\prime }(t)}{t}\right) ^{\frac{\alpha }{\delta }%
-1}h^{\prime \prime }(t)\Biggr] ,
\end{equation*}
where we still denote by $C$ the new constant. If $\mathcal{K}%
_m(t)=h^{\prime\prime}(t)$, then we can conclude by arguing as in \cite%
{Marcellini-Papi 2006}. If otherwise $\mathcal{K}_m(t)=\frac{h^{\prime}(t)}{t%
}$, then it is sufficient to show that 
\begin{equation}  \label{aux lemma1}
\sqrt{\frac{h^{\prime}(t)}{t}}\geq c\left( (h^{\prime\prime}(t))^{\frac{%
\alpha}{\delta}}+(h^{\prime\prime}(t))^{\frac{1}{\delta}}\right)
\end{equation}
which holds by the assumption on $h$ (\ref{the function h(t)}). In fact, if $%
h^{\prime\prime}(t)\geq 1$, then (\ref{aux lemma1}) is equivalent to 
\begin{equation*}
\sqrt{\frac{h^{\prime}(t)}{t}}\geq c(h^{\prime\prime}(t))^{\frac{\alpha}{%
\delta}}
\end{equation*}
and since (\ref{the function h(t)}) holds and $h^{\prime}(t)/t\leq
h^{\prime\prime}(t)$ there exists a constant $c$ such that $%
h^{\prime\prime}(t)\leq\left(\frac{h^{\prime}(t)}{t}\right)^{\alpha}$ for
every $\alpha>1$. Since $\frac{2\alpha}{\delta}>1$, (\ref{aux lemma1})
holds. The other case, $h^{\prime\prime}(t)\leq 1$ can be treated with a
similar argument. Therefore, (\ref{MP_lemma_1}) is proved and then (\ref%
{sec_der_lemma_2}) is proved too.
\end{proof}

\begin{remark}
\label{remark_1} We will also use the following inequality, which is implied
by (\ref{sec_der_lemma_2}): 
\begin{equation}  \label{sec_der_lemma_3}
1+\int_{1}^{t}(s-1)^{\gamma}\sqrt{\mathcal{K}_m(s)}\,ds\geq C\left[ 1+\left( 
\frac{(t-1)^{\gamma+1}t^{-\beta}}{\gamma+1}\right)^{2^{\ast}\sigma} \mathcal{%
K}_M(t) \right]^{\frac{1}{2^{\ast}\sigma}},
\end{equation}
for every $\sigma$ with $\frac{2\alpha }{2^{\ast}(2-\alpha) }\leq \sigma
\leq 1$ and for every $\gamma \geq 1$.
\end{remark}

Let us now treat the case of $0\leq \gamma\leq 1$.

\begin{lemma}
\label{growth_lemma caso gamma piccolo} Let $h$ satisfy (\ref{the function
h(t)}) and let $\mathcal{K}_M, \mathcal{K}_m$ be the functions defined in (%
\ref{def_K}), (\ref{def_K}). Then, for every $\sigma $ with $\frac{2\alpha }{%
2^{\ast }(2-\alpha )}\leq \sigma \leq 1$ and for every $\gamma \in \left[ 0,1%
\right] $ there exists a constant $C$ (depending on $\alpha $) such that,
for every $t\geq 1$, 
\begin{equation}
1+\int_{1}^{t}(s-1)s^{\gamma -1}\sqrt{\mathcal{K}_m(s)}\,ds\geq C\left[
1+\left( \frac{(t-1)^{\gamma +1}t^{-\beta }}{\gamma +1}\right) ^{2^{\ast }}%
\mathcal{K}_M^{\frac{1}{\sigma }}(t)\right] ^{\frac{1}{2^{\ast }}}.
\label{sec_der_lemma_2 caso gamma piccolo}
\end{equation}
\end{lemma}

\begin{proof}
By arguing as in Lemma \ref{growth_lemma}, all we need to prove is the
following: 
\begin{equation}
1+\int_{1}^{t}(s-1)s^{\gamma -1}\sqrt{\mathcal{K}_m(s)}\,ds
\label{MP_lemma_1 caso gamma piccolo}
\end{equation}%
\begin{equation*}
\geq C\left[ 1+\frac{(t-1)^{\gamma +1}t^{-\beta }}{\gamma +1}\left( \frac{%
h^{\prime }(t)}{t}+\left( \frac{h^{\prime }(t)}{t}\right) ^{\alpha }\right)
^{\frac{1}{\delta }}\right] .
\end{equation*}%
Moreover, since $\gamma <1$ and $t\geq 1$, we have 
\begin{equation*}
1+\frac{(t-1)^{\gamma +1}t^{-\beta }}{\gamma +1}\left( \frac{h^{\prime }(t)}{%
t}+\left( \frac{h^{\prime }(t)}{t}\right) ^{\alpha }\right) ^{\frac{1}{%
\delta }}
\end{equation*}%
\begin{equation*}
\leq C\left( 1+\frac{(t-1)^{2}t^{\gamma -\beta -1}}{\gamma +1}\left( \frac{%
h^{\prime }(t)}{t}+\left( \frac{h^{\prime }(t)}{t}\right) ^{\alpha }\right)
^{\frac{1}{\delta }}\right) ,
\end{equation*}%
where $C$ does not depend on $\gamma $. Then, it is sufficient to prove that 
\begin{equation*}
1+\int_{1}^{t}(s-1)s^{\gamma -1}\sqrt{\mathcal{K}_m(s)}\,ds\geq C\left( 1+%
\frac{(t-1)^{2}t^{\gamma -\beta -1}}{\gamma +1}\left( \frac{h^{\prime }(t)}{t%
}+\left( \frac{h^{\prime }(t)}{t}\right) ^{\alpha }\right) ^{\frac{1}{\delta 
}}\right) .
\end{equation*}

As before, it is sufficient to show the inequality between the derivatives
side to side with respect to $t$, i.e., since $\frac{t-1}{t}< 1$ and $%
(\gamma-\beta-1)\frac{t-1}{t}+2< \gamma+1$, 
\begin{equation*}
\sqrt{\mathcal{K}_m(t)}\geq C t^{-\beta}\left[ \left( \frac{h^{\prime}(t)}{t}%
\right)^{\frac{1}{\delta}}+\left(\frac{h^{\prime}(t)}{t}\right)^{\frac{\alpha%
}{\delta}}+\Biggl(\frac{h^{\prime}(t)}{t}\right)^{\frac{1}{\delta}%
-1}h^{\prime\prime}(t)
\end{equation*}
\begin{equation*}
+\left( \frac{h^{\prime}(t)}{t}\right)^{\frac{\alpha}{\delta}%
-1}h^{\prime\prime}(t)\Biggr].
\end{equation*}
Again, by arguing as in Lemma \ref{growth_lemma} we can conclude the proof.
\end{proof}

\begin{remark}
\label{remark_2} We can argue as in Remark \ref{remark_1} to obtain 
\begin{equation}  \label{sec_der_lemma_4}
1+\int_{1}^{t}(s-1) s^{\gamma-1}\sqrt{\mathcal{K}_m(s)}\,ds\geq C\left[
1+\left( \frac{(t-1)^{\gamma+1}t^{-\beta}}{\gamma+1}\right)^{2^{\ast}\sigma} 
\mathcal{K}_M(t) \right]^{\frac{1}{2^{\ast}\sigma}},
\end{equation}
for every $\sigma$ with $\frac{2\alpha }{2^{\ast}(2-\alpha) }\leq \sigma
\leq 1$ and for every $0\leq \gamma\leq 1$.
\end{remark}

We can resume Lemmata \ref{growth_lemma} and \ref{growth_lemma caso gamma
piccolo} in the following lemma, where $\Phi$ is the function defined in (%
\ref{Phi}) if $\gamma\geq 1$ and the function defined in (\ref{Phi - the
case beta small}) if $0\leq\gamma\leq 1$.

\begin{lemma}
\label{lemma for phi for any gamma} Let $h$ satisfy (\ref{the function h(t)}%
) and let $\mathcal{K}_M, \mathcal{K}_m$ be the functions defined in (\ref%
{def_K}), (\ref{def_K_min}). Then, for every $\sigma $ with $\frac{2\alpha }{%
2^{\ast }(2-\alpha )}\leq \sigma \leq 1 $ and for every $\gamma \geq 0$
there exists a constant $C$ (depending on $\alpha $) such that, for every $%
t\geq 1$, 
\begin{equation}
1+\int_{1}^{t}\sqrt{\Phi (s)\mathcal{K}_m(s)}\,ds\geq C\left[ 1+\left( \frac{%
(t-1)^{\frac{\gamma }{2}+1}t^{-\beta }}{\gamma +1}\right) ^{2^{\ast }}%
\mathcal{K}_M^{\frac{1}{\sigma }}(t)\right] ^{\frac{1}{2^{\ast }}}.
\label{sec_der_lemma_2 any gamma}
\end{equation}
\end{lemma}

We will use two consequences of (\ref{sec_der_lemma_2 any gamma}) in section %
\ref{Section: A priori estimates}. The first one is the particular case of $%
\sigma=\frac{1}{\tau}$, with $\tau=(2\vartheta-1)\vartheta$, which is a
compatible value: 
\begin{equation}  \label{lemma_sigma_fissato}
1+\int_{1}^{t}\sqrt{\Phi(s)\mathcal{K}_m(s)}\,ds\geq C\left[ 1+\left( \frac{%
(t-1)^{\frac{\gamma}{2}+1} t^{-\beta }}{\gamma +1}\right) ^{2^{\ast}}%
\mathcal{K}_M^{\tau}(t)\right] ^{\frac{1}{2^{\ast}}}.
\end{equation}%
The second one is essentially the content of Remarks \ref{remark_1} and \ref%
{remark_2} and it is resumed in the following: 
\begin{equation}  \label{sec_der_lemma_31}
1+\int_{1}^{t}\sqrt{\Phi(s)\mathcal{K}_m(s)}\,ds\geq C\left[ 1+\left( \frac{%
(t-1)^{\frac{\gamma}{2}+1}t^{-\beta}}{\gamma+1}\right)^{2^{\ast}\sigma} 
\mathcal{K}_M(t) \right]^{\frac{1}{2^{\ast}\sigma}},
\end{equation}
for any $\frac{2\alpha }{2^{\ast}(2-\alpha) }\leq \sigma \leq 1$, $\gamma
\geq 0$ and every $t\geq 1$.

Next Lemma \ref{h'_lemma} is Lemma $3.2$ of \cite{Marcellini-Papi 2006},
while Lemma \ref{K_lemma} is the generalization of Lemma $3.3$ of the same
paper with $\tau\geq 1$.

\begin{lemma}
\label{h'_lemma} Let $h$ satisfy the right hand side of (\ref{the function
h(t)}). Then there exists a constant $C$, depending on $m_{\alpha }$, $%
h^{\prime }(t_{0})$, $t_{0}$, $\alpha $ such that, for every $t\geq 1$, 
\begin{equation}
h^{\prime }(t)t\leq C\left( 1+h(t)\right) ^{\frac{1}{2-\alpha }}.
\label{h'_lemma_2}
\end{equation}
\end{lemma}

\begin{lemma}
\label{K_lemma} Let $h$ satisfy the right hand side of (\ref{the function
h(t)}) and let $\mathcal{K}_M$ be the functions defined in (\ref{def_K}).
Then, for every $1\leq\tau<\frac{2^{\ast}(2-\alpha)}{2\alpha}$, there exists
a constant $C$ such that for any $\eta $, $1<\eta \leq \frac{n}{n-2} $, 
\begin{equation}
1+\mathcal{K}_M^{\tau}(t)t^{2\tau}\leq C\left( 1+h(t)\right) ^{\eta },
\label{K_lemma_2}
\end{equation}%
for every $t\geq 1$, where $\eta =\eta (\alpha )=\frac{\alpha }{2-\alpha }$
and the constant $C$ depends only on $m_{\alpha }$, $\sup_{0\leq t\leq
1}h^{\prime \prime }(t)$, $\alpha $.
\end{lemma}

\begin{proof}
By the definition of $\mathcal{K}_M$ we have that we have that 
\begin{equation}  \label{K_lemma_3}
\mathcal{K}_M^{\tau}(t)t^{2\tau}\leq \left(\frac{h^{\prime}(t)}{t}%
\right)^{\tau}t^{2\tau}+\left( h^{\prime\prime}(t)\right)^{\tau}t^{2\tau}
=\left( h^{\prime}(t)t\right)^{\tau}+\left(
h^{\prime\prime}(t)t^2\right)^{\tau}
\end{equation}
for every $t\geq 1$. By the right hand side of (\ref{the function h(t)}) and
by Lemma \ref{h'_lemma} we obtain 
\begin{equation}  \label{K_lemma_4}
h^{\prime\prime}(t)t^2\leq m_{\alpha}C \left(1+h(t)\right)^{\frac{1}{2-\alpha%
}}+m_{\alpha}C^{\alpha}\left( 1+h(t)\right)^{\frac{\alpha}{2-\alpha}%
}t^{2-2\alpha}
\end{equation}
\begin{equation*}
\leq C\left(1+h(t)\right)^{\frac{\alpha}{2-\alpha}}.
\end{equation*}
By putting together (\ref{K_lemma_3}) and (\ref{K_lemma_4}) we obtain the
result.
\end{proof}

\section{A-priori estimates\label{Section: A priori estimates}}

By the representation $f\left( x, \xi\right)=g\left( x, \left\vert \xi
\right\vert \right)$, we have 
\begin{equation}  \label{repres_1}
f_{\xi_i^{\alpha}}\left( x, \xi\right)=g_{t}\left( x, \left\vert \xi
\right\vert \right)\frac{\xi_{i}^{\alpha}}{\left\vert \xi\right\vert},
\end{equation}
\begin{equation}  \label{repres_2}
f_{\xi_i^{\alpha}\xi_j^{\beta}}\left( x, \xi\right)=\left( \frac{%
g_{tt}\left(x, \left\vert \xi \right\vert\right)}{\left\vert \xi \right\vert}%
-\frac{g_{t}\left(x, \left\vert \xi \right\vert \right)}{\left\vert \xi
\right\vert^2}\right)\xi_{i}^{\alpha}\xi_{j}^{\beta}+\frac{g_t\left( x,
\left\vert\xi\right\vert \right)}{\left\vert\xi\right\vert}%
\delta_{\xi_i^{\alpha}\xi_j^{\beta}}.
\end{equation}
Thus, the following ellipticity estimates hold:

\begin{equation}  \label{ellip}
\min\left\{ g_{tt}\left( x, \left\vert\xi\right\vert\right), \frac{%
g_{t}\left( x, \left\vert\xi\right\vert\right)}{\left\vert\xi\right\vert}%
\right\}\left\vert\lambda\right\vert^2 \leq \sum_{i, j, \alpha,
\beta}f_{\xi_i^{\alpha}\xi_j^{\beta}}\left(x, \xi\right)
\lambda_i^{\alpha}\lambda_j^{\beta}
\end{equation}
\begin{equation*}
\leq \max\left\{ g_{tt}\left( x, \left\vert\xi\right\vert\right), \frac{%
g_{t}\left( x, \left\vert\xi\right\vert\right)}{\left\vert\xi\right\vert}%
\right\}\left\vert\lambda\right\vert^2,
\end{equation*}

for every $\lambda$, $\xi\in\mathbb{R}^{m\times n}$. Let us define 
\begin{equation}  \label{definition_H_min}
\mathcal{H}_m\left( x, t\right)= \min\left\{ g_{tt}\left( x, t\right), \frac{%
g_{t}\left( x, t\right)}{t}\right\}
\end{equation}
and 
\begin{equation}  \label{definition_H}
\mathcal{H}_M\left( x, t\right)= \max\left\{ g_{tt}\left( x, t\right), \frac{%
g_{t}\left( x, t\right)}{t}\right\},
\end{equation}
then (\ref{ellip}) becomes 
\begin{equation}  \label{ellip2}
\mathcal{H}_m\left(x, \left\vert
\xi\right\vert\right)\left\vert\lambda\right\vert^2 \leq \sum_{i, j, \alpha,
\beta}f_{\xi_i^{\alpha}\xi_j^{\beta}}\left(x, \xi\right)
\lambda_i^{\alpha}\lambda_j^{\beta}\leq \mathcal{H}_M\left( x, \left\vert
\xi\right\vert \right) \left\vert\lambda\right\vert^2,
\end{equation}
for every $\lambda$, $\xi\in\mathbb{R}^{m\times n}$.

We make the following supplementary assumption, which could be later removed
with an approximating procedure, for instance as in Section 5 of \cite%
{Marcellini 1996} and in Section 6 of \cite{Marcellini-Papi 2006}: there
exist two positive constants $N,M$ such that 
\begin{equation}
N\left\vert \lambda \right\vert ^{2}\leq \sum\limits_{i,j,\alpha ,\beta
}f_{\xi _{i}^{\alpha }\xi _{j}^{\beta }}\left( x,\xi \right) \lambda
_{i}^{\alpha }\lambda _{j}^{\beta }\leq M\left\vert \lambda \right\vert ^{2},
\label{suppl_ass}
\end{equation}%
for every $\lambda ,\xi \in \mathbb{R}^{m\times n}$ and for almost every $%
x\in \Omega $. This is equivalent to say that both $\frac{g_{t}}{t}$ and $%
g_{tt}$ are bounded by constants $N$, $M$ for every $t>0$ and for almost
every $x\in \Omega $. This assumption allows us to consider $u$ as a
function of class $W_{\limfunc{loc}}^{1,\infty }\left( \Omega ,\mathbb{R}%
^{m}\right) \cap W_{\limfunc{loc}}^{2,2}\left( \Omega ,\mathbb{R}^{m}\right) 
$.
We denote $B_{\rho }$ and $B_{R}$ balls of radii, respectively, $\rho $ and $%
R$ ($\rho <R$) contained in $\Omega $ and with the same center. In what
follows, we will denote by 
\begin{equation*}
\tilde{B}_R=B_R\cap \left\{ x:\, \vert Du(x)\vert\geq 1\right\}.
\end{equation*}

\begin{lemma}
\label{lemma1} Let $g, h$ respectively satisfy (\ref{main assumptions}) and (%
\ref{the function h(t)}). Suppose that the supplementary condition (\ref%
{suppl_ass}) is satisfied. Let $u\in W_{\limfunc{loc}}^{1,1}\left( \Omega ,%
\mathbb{R}^{m}\right) $ be a local minimizer of (\ref{energy-integral 1}).
Then, under the notation $\tau =\left( 2\vartheta -1\right) \vartheta $, for
every $\rho ,R$ $(0<\rho <R)$ there exists a constant $C$, not depending on $%
m$, and $M,$ such that 
\begin{equation}
\left\Vert Du\right\Vert _{L^{\infty }\left( \tilde{B}_{\rho },\mathbb{R}%
^{m\times n}\right) }^{\left( 1-\beta -\frac{2}{2^{\ast }}\tau \right)
n}\leq \frac{C}{\left( R-\rho \right) ^{n}}\int_{\tilde{B}_{R}}\left(
1+\left\vert Du\right\vert ^{2\tau }\mathcal{K}_{M}^{\tau }\left( \left\vert
Du\right\vert \right) \right) \,dx.  \label{estimates_lemma}
\end{equation}
The constant $C$ depends on $n,\vartheta,\beta,\alpha$.
\end{lemma}

\begin{proof}
Let $u$ be a local minimizer of (\ref{energy-integral 1}). We denote by $%
u=(u^{\alpha})_{\alpha=1,\dots, n}$ its components. By the left hand side of
(\ref{suppl_ass}), $u\in W^{1, 2}\left( \Omega, \mathbb{R}^m\right)$ and by
the right hand side of (\ref{suppl_ass}) $u$ satisfies the Euler's first
variation: 
\begin{equation*}
\int_{\Omega}\sum_{i, \alpha}f_{\xi_i^{\alpha}}\left( x,
Du\right)\varphi_{x_i}^{\alpha}\,dx=0,
\end{equation*}
for every $\varphi=(\varphi^{\alpha})\in W^{1, 2}_0\left(\Omega, \mathbb{R}%
^m\right)$. Using the technique of difference quotients we can prove that $u$
admits second order weak partial derivatives, precisely that $u\in W^{2, 2}_{%
\limfunc{loc}}\left(\Omega, \mathbb{R}^m\right)$ and satisfies the second
variation 
\begin{equation}  \label{euler_2}
\int_{\Omega}\left( \sum\limits_{i, j, \alpha, \beta}
f_{\xi_i^{\alpha}\xi_j^{\beta}}\left(x,
Du\right)\varphi_{x_i}^{\alpha}u_{x_jx_k}^{\beta}+\sum\limits_{i, \alpha}
f_{\xi_i^{\alpha}x_k}\left(x, Du\right)\varphi_{x_i}^{\alpha} \right)\,dx=0,
\end{equation}
for every $k=1,\dots, n$ and for every $\varphi=\left(\varphi^{\alpha}%
\right)\in W^{1, 2}_0\left(\Omega, \mathbb{R}^m\right)$.

Let $R>0$ and $\eta\in C^1_0\left( B_R\right)$. Fixed a positive integer $%
k\leq n$ we consider a test function $\varphi=(\varphi)_{\alpha=1,\dots, n}$
with components defined by 
\begin{equation*}
\varphi^{\alpha}=\eta^2 u_{x_k}^{\alpha}\Phi\left(\left\vert Du\right\vert
\right),
\end{equation*}
where $\Phi:\left[ 0, +\infty\right)\mapsto \left[0, +\infty\right)$ is an
increasing bounded Lipschitz continuous function, such that there exists a
constant $c_{\Phi}\geq 0$ such that 
\begin{equation}  \label{phi_cond}
\Phi^{\prime}(t)t\leq c_{\Phi}\left( 1+ \Phi(t)\right)
\end{equation}
for every $t\geq 1$ and such that $\Phi(t)=0$ if $t\in [0, 1]$. Then, for
the partial derivatives of $\varphi$, it holds 
\begin{equation*}
\varphi_{x_i}^{\alpha}=2\eta\eta_{x_i}u_{x_k}^{\alpha}\Phi\left(\left\vert
Du \right\vert\right)+\eta^2 u_{x_kx_i}^{\alpha}\Phi\left(\left\vert Du
\right\vert\right)+\eta^2 u_{x_k}^{\alpha}\Phi^{\prime}\left(\left\vert Du
\right\vert\right)\left(\left\vert Du \right\vert\right)_{x_i}.
\end{equation*}
From (\ref{euler_2}), we obtain 
\begin{equation}  \label{apriori_1}
0=\int_{\tilde{B}_R}2\eta\Phi\left( \left\vert Du \right\vert\right)\sum_{i,
j, \alpha, \beta} f_{\xi_i^{\alpha}\xi_j^{\beta}}\left(x,
Du\right)\eta_{x_i} u_{x_jx_k}^{\beta}u_{x_k}^{\alpha}\,dx
\end{equation}
\begin{equation*}
+\int_{\tilde{B}_R}\eta^2\Phi\left( \left\vert Du \right\vert\right)\sum_{i,
j, \alpha, \beta} f_{\xi_i^{\alpha}\xi_j^{\beta}}\left(x,
Du\right)u_{x_jx_k}^{\beta}u_{x_ix_k}^{\alpha}\,dx
\end{equation*}
\begin{equation*}
+\int_{\tilde{B}_R}\eta^2\Phi^{\prime}\left( \left\vert Du
\right\vert\right)\sum_{i, j, \alpha,
\beta}f_{\xi_i^{\alpha}\xi_j^{\beta}}\left(x,
Du\right)u_{x_k}^{\alpha}u_{x_jx_k}^{\beta}\left( \left\vert Du
\right\vert\right)_{x_i}\,dx
\end{equation*}
\begin{equation*}
+\int_{\tilde{B}_R}2\eta\Phi\left( \left\vert Du \right\vert\right)\sum_{i,
\alpha}f_{\xi_i^{\alpha}x_k}\left(x, Du\right)\eta_{x_i}u_{x_k}^{\alpha}\,dx
\end{equation*}
\begin{equation*}
+\int_{\tilde{B}_R}\eta^2\Phi\left( \left\vert Du \right\vert\right)\sum_{i,
\alpha}f_{\xi_i^{\alpha}x_k}\left(x, Du\right)u_{x_ix_k}^{\alpha}\,dx
\end{equation*}
\begin{equation*}
+\int_{\tilde{B}_R}\eta^2\Phi^{\prime}\left( \left\vert Du
\right\vert\right)\sum_{i, \alpha}f_{\xi_i^{\alpha}x_k}\left(x,
Du\right)u_{x_k}^{\alpha}\left( \left\vert Du \right\vert\right)_{x_i}\,dx
\end{equation*}
\begin{equation*}
=:I_1+I_2+I_3+I_4+I_5+I_6.
\end{equation*}

We start estimating $I_{1}$ in (\ref{apriori_1}) with the Cauchy-Schwarz
inequality and the Young's inequality $2ab\leq \frac{1}{2}a^{2}+2b^{2}$: 
\begin{equation}
\left\vert I_{1}\right\vert \leq \int_{\tilde{B}_{R}}2\Phi \left( \left\vert
Du\right\vert \right) \left[ \eta ^{2}\sum_{i,j,\alpha ,\beta }f_{\xi
_{i}^{\alpha }\xi _{j}^{\beta }}\left( x,Du\right) u_{x_{j}x_{k}}^{\beta
}u_{x_{i}x_{k}}^{\alpha }\right] ^{1/2}  \label{I_1_estimate}
\end{equation}%
\begin{equation*}
\cdot \left[ \sum_{i,j,\alpha ,\beta }f_{\xi _{i}^{\alpha }\xi _{j}^{\beta
}}\left( x,Du\right) \eta _{x_{i}}\eta _{x_{j}}u_{x_{k}}^{\alpha
}u_{x_{k}}^{\beta }\right] ^{1/2}\,dx
\end{equation*}%
\begin{equation*}
\leq \int_{\tilde{B}_{R}}\Phi \left( \left\vert Du\right\vert \right) \Biggl[%
\frac{1}{2}\eta ^{2}\sum_{i,j,\alpha ,\beta }f_{\xi _{i}^{\alpha }\xi
_{j}^{\beta }}\left( x,Du\right) u_{x_{j}x_{k}}^{\beta
}u_{x_{i}x_{k}}^{\alpha }
\end{equation*}%
\begin{equation*}
+2\sum_{i,j,\alpha ,\beta }f_{\xi _{i}^{\alpha }\xi _{j}^{\beta }}\eta
_{x_{i}}\eta _{x_{j}}u_{x_{k}}^{\alpha }u_{x_{k}}^{\beta }\Biggr]\,dx.
\end{equation*}%
From (\ref{apriori_1}) and (\ref{I_1_estimate}) we obtain 
\begin{equation}
\frac{1}{2}I_{2}+I_{3}+I_{4}+I_{5}+I_{6}\leq 2\int_{\tilde{B}_{R}}\Phi
\left( \left\vert Du\right\vert \right) \sum_{i,j,\alpha ,\beta }f_{\xi
_{i}^{\alpha }\xi _{j}^{\beta }}\eta _{x_{i}}\eta _{x_{j}}u_{x_{k}}^{\alpha
}u_{x_{k}}^{\beta }\,dx.  \label{apriori_2}
\end{equation}%
We use the expression of the second derivatives of $f$ to estimate $I_{3}$.
Since 
\begin{equation}
\left( \left\vert Du\right\vert \right) _{x_{i}}=\frac{1}{\left\vert
Du\right\vert }\sum_{\alpha ,k}u_{x_{i}x_{k}}^{\alpha }u_{x_{k}}^{\alpha },
\label{sum_motiv}
\end{equation}%
it is natural to sum over $k$ and we observe that 
\begin{equation}
\sum_{k}\sum_{i,j,\alpha ,\beta }f_{\xi _{i}^{\alpha }\xi _{j}^{\beta
}}\left( x,Du\right) u_{x_{k}}^{\alpha }u_{x_{j}x_{k}}^{\beta }\left(
\left\vert Du\right\vert \right) _{x_{i}}  \label{I_3_estimate_1}
\end{equation}%
\begin{equation*}
=\left( \frac{g_{tt}\left( x,\left\vert Du\right\vert \right) }{\left\vert
Du\right\vert ^{2}}-\frac{g_{t}\left( x,\left\vert Du\right\vert \right) }{%
\left\vert Du\right\vert ^{3}}\right) \sum_{i,j,k,\alpha ,\beta
}u_{x_{i}}^{\alpha }u_{x_{j}}^{\beta }u_{x_{j}x_{k}}^{\beta
}u_{x_{k}}^{\alpha }\left( \left\vert Du\right\vert \right) _{x_{i}}
\end{equation*}%
\begin{equation*}
+\frac{g_{t}\left( x,\left\vert Du\right\vert \right) }{\left\vert
Du\right\vert }\sum_{i,k,\alpha }u_{x_{i}x_{k}}^{\alpha }u_{x_{k}}^{\alpha
}\left( \left\vert Du\right\vert \right) _{x_{i}}
\end{equation*}%
\begin{equation*}
=\left( \frac{g_{tt}\left( x,\left\vert Du\right\vert \right) }{\left\vert
Du\right\vert }-\frac{g_{t}\left( x,\left\vert Du\right\vert \right) }{%
\left\vert Du\right\vert ^{2}}\right) \sum_{i,k,\alpha }u_{x_{i}}^{\alpha
}\left( \left\vert Du\right\vert \right) _{x_{i}}u_{x_{k}}^{\alpha }\left(
\left\vert Du\right\vert \right) _{x_{k}}
\end{equation*}%
\begin{equation*}
+g_{t}\left( x,\left\vert Du\right\vert \right) \sum_{i}\left( \left\vert
Du\right\vert \right) _{x_{i}}^{2}
\end{equation*}%
\begin{equation*}
=\left( \frac{g_{tt}\left( x,\left\vert Du\right\vert \right) }{\left\vert
Du\right\vert }-\frac{g_{t}\left( x,\left\vert Du\right\vert \right) }{%
\left\vert Du\right\vert ^{2}}\right) \sum_{\alpha }\left[
\sum_{i}u_{x_{i}}^{\alpha }\left\vert Du\right\vert _{x_{i}}\right] ^{2}
\end{equation*}%
\begin{equation*}
+g_{t}\left( x,\left\vert Du\right\vert \right) \left\vert D\left(
\left\vert Du\right\vert \right) \right\vert ^{2}.
\end{equation*}

Now, if we denote with $\tilde{I}_{s}$ the sum over $k$ of $I_{s}$, for $%
s=1,\dots ,6$, we have that 
\begin{equation}
\tilde{I}_{3}=\int_{\tilde{B}_{R}}\eta ^{2}\Phi ^{\prime }\left( \left\vert
Du\right\vert \right) \Biggl[\left( \frac{g_{tt}\left( x,\left\vert
Du\right\vert \right) }{\left\vert Du\right\vert }-\frac{g_{t}\left(
x,\left\vert Du\right\vert \right) }{\left\vert Du\right\vert ^{2}}\right)
\label{I_3_estimate_2}
\end{equation}%
\begin{equation*}
\cdot \sum_{\alpha }\left[ \sum_{i}u_{x_{i}}^{\alpha }\left( \left\vert
Du\right\vert \right) _{x_{i}}\right] ^{2}+g_{t}\left( x,\left\vert
Du\right\vert \right) \left\vert D\left( \left\vert Du\right\vert \right)
\right\vert ^{2}\Biggl]\,dx
\end{equation*}%
\begin{equation*}
=\int_{\tilde{B}_{R}}\eta ^{2}\Phi ^{\prime }\left( \left\vert Du\right\vert
\right) \Biggl[\frac{g_{tt}\left( x,\left\vert Du\right\vert \right) }{%
\left\vert Du\right\vert }\sum_{\alpha }\left( \sum_{i}u_{x_{i}}^{\alpha
}\left( \left\vert Du\right\vert \right) _{x_{i}}\right) ^{2}+g_{t}\left(
x,\left\vert Du\right\vert \right) \left\vert D\left( \left\vert
Du\right\vert \right) \right\vert ^{2}
\end{equation*}%
\begin{equation*}
-\frac{g_{t}\left( x,\left\vert Du\right\vert \right) }{\left\vert
Du\right\vert ^{2}}\sum_{\alpha }\left( \sum_{i}u_{x_{i}}^{\alpha }\left(
\left\vert Du\right\vert \right) _{x_{i}}\right) ^{2}\Biggl]\,dx.
\end{equation*}%
Since, by Cauchy-Schwarz inequality we get 
\begin{equation}
\sum_{\alpha }\left( \sum_{i}u_{x_{i}}^{\alpha }\left( \left\vert
Du\right\vert \right) _{x_{i}}\right) ^{2}\leq \sum_{i,\alpha }\left(
u_{x_{i}}^{\alpha }\right) ^{2}\sum_{i}\left( \left\vert Du\right\vert
\right) _{x_{i}}^{2}\leq \left\vert Du\right\vert ^{2}\left\vert D\left(
\left\vert Du\right\vert \right) \right\vert ^{2},  \label{sec_C_S_ineq}
\end{equation}%
then we can conclude that 
\begin{equation}
\tilde{I}_{3}\geq \int_{\tilde{B}_{R}}\eta ^{2}\Phi ^{\prime }\left(
\left\vert Du\right\vert \right) \frac{g_{tt}\left( x,\left\vert
Du\right\vert \right) }{\left\vert Du\right\vert }\sum_{\alpha }\left(
\sum_{i}u_{x_{i}}^{\alpha }\left( \left\vert Du\right\vert \right)
_{x_{i}}\right) ^{2}\,dx\geq 0.  \label{I_3_estimate_final}
\end{equation}

Now, we consider the term $\frac{1}{2}I_2$ in inequality (\ref{apriori_2}).
From the ellipticity condition (\ref{ellip2}) 
\begin{equation}  \label{I_2_estimate}
\vert \tilde{I}_2\vert\geq \int_{\tilde{B}_R}\eta^2 \Phi\left(\left\vert
Du\right\vert\right)\mathcal{H}_m\left( x, \left\vert Du\right\vert
\right)\left\vert D^2u\right\vert^2\,dx.
\end{equation}

By using (\ref{I_3_estimate_final}), (\ref{I_2_estimate}) and by summing
over $k$ in formula (\ref{apriori_2}), we obtain 
\begin{equation}  \label{apriori_3}
\frac{1}{2}\int_{\tilde{B}_R}\eta^2 \Phi\left(\left\vert Du\right\vert\right)%
\mathcal{H}_m\left( x, \left\vert Du\right\vert\right) \left\vert D^2
u\right\vert^2\,dx
\end{equation}
\begin{equation*}
\leq \vert\tilde{I}_4\vert+\vert\tilde{I}_5\vert+\vert\tilde{I}%
_6\vert+2\int_{\tilde{B}_R}\Phi\left( \left\vert Du \right\vert\right)
\sum_{i, j, k, \alpha, \beta}
f_{\xi_i^{\alpha}\xi_j^{\beta}}\eta_{x_i}\eta_{x_j}u_{x_k}^{\alpha}u_{x_k}^{%
\beta}\,dx.
\end{equation*}

Consider now $\tilde{I}_4$. Since 
\begin{equation}  \label{motiv_2}
\left\vert f_{\xi_ix_k}\left(x, \xi\right)\right\vert\leq \left\vert
g_{tx_k}\left( x, \left\vert\xi\right\vert\right)\right\vert
\end{equation}
for almost every $x\in\Omega$ and for every $\xi\in\mathbb{R}^{m\times n}$,
then, by the third assumption of (\ref{main assumptions}), for $%
\tau\in\left( 1, \vartheta\right)$ to be fixed, 
\begin{equation}  \label{I_4_estimate}
\vert \tilde{I}_4\vert=\left\vert \int_{\tilde{B}_R}
2\eta\Phi\left(\left\vert Du\right\vert\right)\sum_{i, k,
\alpha}f_{\xi_{i}^{\alpha} x_k}\left( x, \left\vert
Du\right\vert\right)\eta_{x_i}u_{x_k}^{\alpha}\,dx\right\vert
\end{equation}
\begin{equation*}
\leq \int_{\tilde{B}_R}2\eta\left\vert D\eta\right\vert \Phi\left(\left\vert
Du\right\vert\right) \sum_{i} \left\vert g_{tx_k}\left( x, \left\vert
Du\right\vert\right)\right\vert \left\vert Du\right\vert\,dx
\end{equation*}
\begin{equation*}
\leq c\int_{\tilde{B}_R}2\eta\left\vert D\eta\right\vert
\Phi\left(\left\vert Du\right\vert\right) \mathcal{H}_m^{\vartheta}\left( x,
\left\vert Du\right\vert\right) \left\vert Du\right\vert^{1+\vartheta}\,dx
\end{equation*}
\begin{equation*}
\leq c\int_{\tilde{B}_R}2\eta\left\vert D\eta\right\vert
\Phi\left(\left\vert Du\right\vert\right) \mathcal{H}_M^{\vartheta}\left( x,
\left\vert Du\right\vert\right) \left\vert Du\right\vert^{1+\vartheta}\,dx.
\end{equation*}

By the Young's inequality and again the third condition of (\ref{main
assumptions}), we obtain 
\begin{equation}  \label{I_5_estimate}
\vert \tilde{I}_5\vert = \left\vert \int_{\tilde{B}_R}\eta^2
\Phi\left(\left\vert Du\right\vert\right)\sum_{i, k,
\alpha}f_{\xi_{i}^{\alpha} x_k}\left( x, \left\vert Du\right\vert\right)
u_{x_ix_k}^{\alpha}\,dx\right\vert
\end{equation}
\begin{equation*}
\leq c\int_{\tilde{B}_R}\eta^2\Phi\left(\left\vert
Du\right\vert\right)\sum_{k} \left\vert g_{tx_k}\left( x, \left\vert
Du\right\vert\right)\right\vert \left\vert D^2u\right\vert\,dx
\end{equation*}
\begin{equation*}
\leq c\int_{\tilde{B}_R}\eta^2\Phi\left(\left\vert Du\right\vert\right) 
\mathcal{H}_m^{\vartheta}\left( x, \left\vert Du\right\vert\right)
\left\vert Du\right\vert^{\vartheta}\left\vert D^2u\right\vert\,dx
\end{equation*}
\begin{equation*}
\leq c\int_{\tilde{B}_R}\eta^2\Phi\left(\left\vert Du\right\vert\right)
\left( \mathcal{H}_m \left( x, \left\vert Du\right\vert\right)\left\vert
D^2u\right\vert^2\right)^{1/2} \left( \mathcal{H}_m^{2\vartheta-1}\left( x,
\left\vert Du\right\vert\right) \left\vert
Du\right\vert^{2\vartheta}\right)^{1/2}\,dx
\end{equation*}
\begin{equation*}
\leq c\varepsilon\int_{\tilde{B}_R} \eta^2\Phi\left(\left\vert
Du\right\vert\right)\mathcal{H}_m \left( x, \left\vert
Du\right\vert\right)\left\vert D^2u\right\vert^2\,dx
\end{equation*}
\begin{equation*}
+\frac{c}{4\varepsilon}\int_{\tilde{B}_R} \eta^2\Phi\left(\left\vert
Du\right\vert\right)\mathcal{H}_M^{2\vartheta-1}\left( x, \left\vert
Du\right\vert\right) \left\vert Du\right\vert^{2\vartheta}\,dx.
\end{equation*}

We choose $\varepsilon$ sufficiently small to absorb the first integral in
the last inequality of formula (\ref{I_5_estimate}) in the left hand side of
(\ref{apriori_3}). Similarly 
\begin{equation}  \label{I_6_estimate}
\vert \tilde{I}_6\vert =\left\vert \int_{\tilde{B}_R}\eta^2\Phi^{\prime}%
\left(\left\vert Du\right\vert\right)\sum_{i, k, \alpha} f_{\xi_{i}^{\alpha}
x_k}\left( x, \left\vert Du\right\vert\right)u_{x_k}^{\alpha} \left(
\left\vert Du\right\vert\right)_{x_i}\,dx\right\vert
\end{equation}
\begin{equation*}
\leq c\int_{\tilde{B}_R} \eta^2\Phi^{\prime}\left(\left\vert
Du\right\vert\right) \sum_{i, k}\left\vert g_{tx_k}\left( x, \left\vert
Du\right\vert\right)\right\vert \left\vert Du\right\vert \left(\left\vert
Du\right\vert\right)_{x_i}\,dx
\end{equation*}
\begin{equation*}
\leq c\int_{\tilde{B}_R}\eta^2\Phi^{\prime}\left(\left\vert
Du\right\vert\right)\left\vert Du\right\vert \left( \mathcal{H}_m\left( x,
\left\vert Du\right\vert \right)\sum_{i}\left(\left\vert Du\right\vert
\right)_{x_i}^2\right)^{1/2}
\end{equation*}
\begin{equation*}
\cdot\left( \mathcal{H}_m^{2\vartheta-1}\left( x, \left\vert Du\right\vert
\right)\left\vert Du\right\vert^{2\vartheta} \right)^{1/2}\,dx
\end{equation*}
\begin{equation*}
\leq c\varepsilon \int_{\tilde{B}_R} \eta^2\Phi^{\prime}\left(\left\vert
Du\right\vert\right)\left\vert Du\right\vert \mathcal{H}_m\left( x,
\left\vert Du\right\vert \right)\sum_{i}\left(\left\vert Du\right\vert
\right)_{x_i}^2\,dx
\end{equation*}
\begin{equation*}
+\frac{c}{4\varepsilon}\int_{\tilde{B}_R}
\eta^2\Phi^{\prime}\left(\left\vert Du\right\vert\right)\left\vert
Du\right\vert \mathcal{H}_m^{2\vartheta-1}\left( x, \left\vert Du\right\vert
\right)\left\vert Du\right\vert^{2\vartheta}\,dx.
\end{equation*}
\begin{equation*}
\leq c_{\Phi}\varepsilon \int_{\tilde{B}_R} \eta^2\left(
1+\Phi\left(\left\vert Du\right\vert\right)\right)\mathcal{H}_m\left( x,
\left\vert Du\right\vert \right)\left\vert D^2u\right\vert^2\,dx
\end{equation*}
\begin{equation*}
+\frac{c_{\Phi}}{4\varepsilon}\int_{\tilde{B}_R} \eta^2\left( 1+
\Phi\left(\left\vert Du\right\vert\right)\right) \mathcal{H}%
_M^{2\vartheta-1}\left( x, \left\vert Du\right\vert \right)\left\vert
Du\right\vert^{2\vartheta}\,dx
\end{equation*}
where in the last inequality we have used (\ref{phi_cond}) and (\ref%
{sec_C_S_ineq}).

We add in both sides of (\ref{apriori_3}) the quantity 
\begin{equation*}
\int_{\tilde{B}_R}\eta^2 \mathcal{H}_m(x, \vert Du\vert)\left\vert
D^2u\right\vert^2\,dx
\end{equation*}
and we get

\begin{equation}  \label{apriori_4}
\int_{\tilde{B}_R}\eta^2 \left(1+\Phi\left(\left\vert
Du\right\vert\right)\right) \mathcal{H}_m\left( x, \left\vert
Du\right\vert\right) \left\vert D^2u\right\vert^2\,dx
\end{equation}
\begin{equation*}
\leq c\int_{\tilde{B}_R}2\eta\left\vert D\eta\right\vert
\Phi\left(\left\vert Du\right\vert\right) \mathcal{H}_M^{\vartheta}\left( x,
\left\vert Du\right\vert\right) \left\vert Du\right\vert^{1+\vartheta}\,dx
\end{equation*}
\begin{equation*}
+c\int_{\tilde{B}_R} \eta^2\Phi\left(\left\vert Du\right\vert\right)\mathcal{%
H}_M^{2\vartheta-1}\left( x, \left\vert Du\right\vert\right) \left\vert
Du\right\vert^{2\vartheta}\,dx
\end{equation*}
\begin{equation*}
+c_{\Phi}\varepsilon \int_{\tilde{B}_R} \eta^2\left(1+\Phi\left(\left\vert
Du\right\vert\right)\right) \mathcal{H}_m\left( x, \left\vert Du\right\vert
\right) \left\vert D^2 u\right\vert^2\,dx
\end{equation*}
\begin{equation*}
+\frac{c_{\Phi}}{4\varepsilon}\int_{\tilde{B}_R}
\eta^2\left(1+\Phi\left(\left\vert Du\right\vert\right)\right) \mathcal{H}%
_M^{2\vartheta-1}\left( x, \left\vert Du\right\vert \right)\left\vert
Du\right\vert^{2\vartheta}\,dx.
\end{equation*}
\begin{equation*}
+c \int_{\tilde{B}_R}\Phi\left( \left\vert Du \right\vert\right) \sum_{i, j,
k, \alpha, \beta}
f_{\xi_i^{\alpha}\xi_j^{\beta}}\eta_{x_i}\eta_{x_j}u_{x_k}^{\alpha}u_{x_k}^{%
\beta}\,dx.
\end{equation*}
Then, choosing $\varepsilon$ sufficiently and using the ellipticity
condition (\ref{ellip2}) in (\ref{apriori_4}) and by also using the
Cauchy-Schwarz inequality, we get

\begin{equation}  \label{apriori_5}
\int_{\tilde{B}_R}\eta^2 \Phi\left(\left\vert Du\right\vert\right)\mathcal{H}%
_m\left( x, \left\vert Du\right\vert\right) \left\vert D\left( \left\vert
Du\right\vert\right)\right\vert^2\,dx
\end{equation}
\begin{equation*}
\leq c\int_{\tilde{B}_R}2\eta\left\vert D\eta\right\vert
\Phi\left(\left\vert Du\right\vert\right)\mathcal{H}_M^{\vartheta}\left( x,
\left\vert Du\right\vert\right) \left\vert Du\right\vert^{1+\vartheta}\,dx
\end{equation*}
\begin{equation*}
+c(1+c_{\Phi})\int_{\tilde{B}_R} \eta^2\left( 1+ \Phi\left(\left\vert
Du\right\vert\right)\right) \mathcal{H}_M^{2\vartheta-1}\left( x, \left\vert
Du\right\vert\right) \left\vert Du\right\vert^{2\vartheta}\,dx
\end{equation*}
\begin{equation*}
+c \int_{\tilde{B}_R}\left\vert D\eta\right\vert^2 \Phi\left( \left\vert Du
\right\vert\right) \mathcal{H}_M^{2\vartheta-1}\left( x, \left\vert
Du\right\vert\right)\left\vert Du\right\vert^2 \,dx.
\end{equation*}
Then, since $\left\vert Du\right\vert\geq 1$ in $\tilde{B}_R$, (\ref%
{apriori_5}) becomes

\begin{equation}  \label{last_before_sobolev}
\int\limits_{\tilde{B}_R} \eta^2\Phi\left(\vert Du \vert\right)\mathcal{H}%
_m\left( x, \left\vert Du \right\vert \right)\left\vert D\left(\left\vert
Du\right\vert\right) \right\vert^2\,dx
\end{equation}
\begin{equation*}
\leq c\left( 1+c_{\Phi}\right)\int_{\tilde{B}_R}\left(\eta^2+\left\vert
D\eta\right\vert^2\right)\left( 1+ \Phi\left(\left\vert
Du\right\vert\right)\right)\bigl(\mathcal{H}_M^{\vartheta}\left( x,
\left\vert Du\right\vert\right)
\end{equation*}
\begin{equation*}
+\mathcal{H}_M^{2\vartheta-1}\left( x, \left\vert Du\right\vert\right)\bigr)%
\left\vert Du\right\vert^{2\vartheta}\,dx.
\end{equation*}
Now we deal with the problem of a non bounded and non Lipschitz continuous $%
\Phi$. We can approximate $\Phi$ with a sequence of functions $\Phi_r$, each
of them equal to $\Phi$ in the interval $[0, r]$ and extended continuously
to $[r, +\infty)$ with the constant value $\Phi(r)$. We insert $\Phi_r$ in (%
\ref{last_before_sobolev}) and we go to the limit as $r\to\infty$ by the
monotone convergence theorem. So we obtain that (\ref{last_before_sobolev})
is true for every $\Phi$ positive, increasing and local Lipschitz continuous
function in $[0, +\infty)$.

We apply (\ref{main assumptions}) and we find that the function $\mathcal{K}%
_M$ defined in (\ref{def_K}) satisfies the condition 
\begin{equation*}
\mathcal{H}_M(x, t)\leq M_{\vartheta} \mathcal{K}_M^{\vartheta}(t)
\end{equation*}
for every $t\geq 1$ and for almost every $x\in B_R$. Similarly, the function 
$\mathcal{K}_m$ defined in (\ref{def_K_min}) satisfies 
\begin{equation*}
\mathcal{H}_m(x, t)\geq m \mathcal{K}_m(t)
\end{equation*}
for every $t\geq 1$ and for almost every $x\in B_R$. We define 
\begin{equation*}
G(t)=1+\int_0^t \sqrt{\Phi(s) \mathcal{K}_m(s)}\,ds.
\end{equation*}
By the H\"older inequality, since $\Phi$ is increasing and $h^{\prime}(0)=0$%
, we get the following inequalities concerning $G$ and its derivatives:

\begin{equation*}
G^{2}(t)=\left( 1+\int_{0}^{t}\sqrt{\Phi (s)\mathcal{K}_{m}(s)}\,ds\right)
^{2}\leq 2+2\Phi (t)t\int_{0}^{t}h^{\prime \prime }(s)\,ds\leq 2+2\Phi
(t)th^{\prime }(t)
\end{equation*}%
\begin{equation*}
\leq 2\left( 1+\Phi (t)\mathcal{K}_{M}(t)t^{2}\right)
\end{equation*}%
and 
\begin{equation*}
\left\vert D\left( \eta G\left( \left\vert Du\right\vert \right) \right)
\right\vert ^{2}\leq 2\left\vert D\eta \right\vert ^{2}\left\vert G\left(
\left\vert Du\right\vert \right) \right\vert ^{2}+2\eta ^{2}\left\vert
G^{\prime }\left( \left\vert Du\right\vert \right) \right\vert
^{2}\left\vert D\left( \left\vert Du\right\vert \right) \right\vert ^{2}.
\end{equation*}%
We get 
\begin{equation}
\int_{\tilde{B}_{R}}\left\vert D\left( \eta G\left( \left\vert Du\right\vert
\right) \right) \right\vert ^{2}\,dx  \label{before_sobolev_2}
\end{equation}%
\begin{equation*}
\leq c\left( 1+c_{\Phi }\right) \int_{\tilde{B}_{R}}\left( \eta
^{2}+\left\vert D\eta \right\vert ^{2}\right) \left( 1+\Phi \left(
\left\vert Du\right\vert \right) \right) \bigl(\mathcal{K}_{M}^{\vartheta
^{2}}\left( x,\left\vert Du\right\vert \right)
\end{equation*}%
\begin{equation*}
\quad +\mathcal{K}_{M}^{(2\vartheta -1)\vartheta }\left( \left\vert
Du\right\vert \right) \bigr)\left\vert Du\right\vert ^{2\vartheta
}\,dx\,+\,c_{2}\int_{\tilde{B}_{R}}\left\vert D\eta \right\vert ^{2}\left(
1+\Phi \left( \left\vert Du\right\vert \right) \mathcal{K}_{M}\left(
\left\vert Du\right\vert \right) \left\vert Du\right\vert ^{2}\right) \,dx.
\end{equation*}%
We recall that $\vartheta $ satisfies the following conditions%
\begin{equation*}
1\leq \left( 2\vartheta -1\right) \vartheta <\left( 1-\beta \right) \frac{%
2^{\ast }}{2}
\end{equation*}
and we introduce the notation 
\begin{equation}
\tau =(2\vartheta -1)\vartheta .  \label{tau def}
\end{equation}%
Therefore $1\leq \tau <\left( 1-\beta \right) \frac{2^{\ast }}{2}$and $\tau
\geq \vartheta ^{2}$. Then, (\ref{before_sobolev_2}) leads to 
\begin{equation}
\int_{\tilde{B}_{R}}\left\vert D\left( \eta G\left( \left\vert Du\right\vert
\right) \right) \right\vert ^{2}\,dx  \label{before_sobolev_3}
\end{equation}%
\begin{equation*}
\leq c\left( 1+c_{\Phi }\right) \int_{\tilde{B}_{R}}\left( \eta
^{2}+\left\vert D\eta \right\vert ^{2}\right) \biggl[1+\Phi \left(
\left\vert Du\right\vert \right) \mathcal{K}_{M}^{\tau }\left( \left\vert
Du\right\vert \right) \left\vert Du\right\vert ^{2\tau }\biggr]\,dx.
\end{equation*}%
We apply the Sobolev inequality to get

\begin{equation}  \label{before Phi}
\biggl[\int_{\tilde{B}_R}\left\vert \eta G\left(\left\vert
Du\right\vert\right) \right\vert^{2^{\ast}}\,dx \biggr]^{2/2^{\ast}}
\end{equation}
\begin{equation*}
\leq c ( 1+ c_{\Phi})\int_{\tilde{B}_R}\left(\eta^2+\left\vert
D\eta\right\vert^2\right)\left( 1+\Phi \left(\left\vert Du
\right\vert\right) \mathcal{K}_M^{\tau}\left( \left\vert Du\right\vert
\right) \left\vert Du\right\vert^{2\tau} \right)\,dx.
\end{equation*}

We choose $\Phi$ equal to the function defined in (\ref{Phi}) if $\gamma\geq
1$. Then, since $t^{\gamma-2}\leq (t-1)^{\gamma-2}$ for every $\gamma\in%
\left[0,2\right]$ and $t\geq 1$, by Lemmata \ref{Lemma for Phi} and \ref%
{Lemma for Phi - the case beta small}, (\ref{before Phi}) becomes

\begin{equation*}
\biggl[\int_{\tilde{B}_R}\left\vert \eta G\left(\left\vert
Du\right\vert\right) \right\vert^{2^{\ast}}\,dx \biggr]^{2/2^{\ast}}
\end{equation*}
\begin{equation*}
\leq c(1+\gamma)\int_{\tilde{B}_R} \left( \eta^2+\left\vert
D\eta\right\vert^2\right) \left( 1+\left(\left\vert
Du\right\vert-1\right)^{\gamma} \mathcal{K}_M^{\vartheta}\left( \left\vert
Du\right\vert \right) \left\vert Du\right\vert^{2\vartheta}\right)\, dx.
\end{equation*}
Since $\mathcal{K}_m$ satisfies (\ref{lemma_sigma_fissato}), we get

\begin{equation}
\left[ \int_{\tilde{B}_{R}}\eta ^{2^{\ast }}\left( 1+\left( \left\vert
Du\right\vert -1\right) ^{\frac{2^{\ast }}{2}(\gamma +2)}\left\vert
Du\right\vert ^{-2^{\ast }\beta }\mathcal{K}_M^{\tau }\left( \left\vert
Du\right\vert \right) \right) \,dx\right] ^{2/2^{\ast }}
\label{after_sobolev_2}
\end{equation}%
\begin{equation*}
\leq c(\gamma +1)^{3}\int_{\tilde{B}_{R}}\left( \eta ^{2}+\left\vert D\eta
\right\vert ^{2}\right) \left( 1+\left( \left\vert Du\right\vert -1\right)
^{\gamma }\mathcal{K}_M^{\tau }\left( \left\vert Du\right\vert \right)
\left\vert Du\right\vert ^{2\tau }\right) \,dx.
\end{equation*}%
Now, since 
\begin{equation*}
1+\left( t-1\right) ^{\frac{2^{\ast }}{2}(\gamma +2)}t^{-2^{\ast }\beta }%
\mathcal{K}^{\tau }(t)
\end{equation*}%
\begin{equation*}
=1+\left( t-1\right) ^{\frac{2^{\ast }}{2}(\gamma +2)-(2\tau +2^{\ast }\beta
)}\left( t-1\right) ^{2\tau +2^{\ast }\beta }t^{-2^{\ast }\beta }\mathcal{K}%
_M^{\tau }(t)
\end{equation*}%
\begin{equation*}
\geq C\left( 1+\left( t-1\right) ^{\frac{2^{\ast }}{2}(\gamma +2)-(2\tau
+2^{\ast }\beta )}t^{2\tau }\mathcal{K}_M^{\tau }(t)\right) ,
\end{equation*}%
with $C$ not depending on $\gamma $, (\ref{after_sobolev_2}) becomes 
\begin{equation}
\left[ \int_{\tilde{B}_{R}}\eta ^{2^{\ast }}\left( 1+\left( \left\vert
Du\right\vert -1\right) ^{\frac{2^{\ast }}{2}(\gamma +2)-(2\tau +2^{\ast
}\beta )}\mathcal{K}^{\tau }\left( \left\vert Du\right\vert \right)
\left\vert Du\right\vert ^{2\tau }\right) \,dx\right] ^{2/2^{\ast }}
\label{after_sobolev_3}
\end{equation}%
\begin{equation*}
\leq c(\gamma +1)^{3}\int_{\tilde{B}_{R}}\left( \eta ^{2}+\left\vert D\eta
\right\vert ^{2}\right) \left( 1+\left( \left\vert Du\right\vert -1\right)
^{\gamma }\mathcal{K}^{\tau }\left( \left\vert Du\right\vert \right)
\left\vert Du\right\vert ^{2\tau }\right) \,dx.
\end{equation*}%
Let us fix $0<\rho <R$ and take $\eta \equiv 1$ in $\tilde{B}_{\rho }$ and $%
\left\vert D\eta \right\vert \leq \frac{2}{R-\rho }$. Then, fixed $\overline{%
\rho }<\overline{R}$, let us also define the decreasing sequence of radii $%
\left\{ \rho _{i}\right\} $, defined by 
\begin{equation*}
\rho _{i}=\overline{\rho }+\frac{\overline{R}-\overline{\rho }}{2^{i}},
\end{equation*}%
for every $i\geq 0$. We define the sequence $\left\{ \gamma _{i}\right\} $
defined by the recurrence $\gamma _{0}=0$, 
\begin{equation*}
\gamma _{i+1}=\frac{2^{\ast }}{2}\left( \gamma _{i}+2\right) -(2\tau
+2^{\ast }\beta ),
\end{equation*}%
which is non decreasing by the properties of $\beta $ and $\tau$. Then for
every $i\geq 0$ 
\begin{equation}
\left[ \int_{\tilde{B}_{\rho _{i+1}}}\left( 1+\left( \left\vert
Du\right\vert -1\right) ^{\gamma _{i+1}}\mathcal{K}_M^{\tau }\left(
\left\vert Du\right\vert \right) \left\vert Du\right\vert ^{2\tau }\right)
\,dx\right] ^{2/2^{\ast }}  \label{after_sobolev_31}
\end{equation}%
\begin{equation*}
\leq c(\gamma _{i}+1)^{3}\left( \frac{2^{i+1}}{\overline{R}-\overline{\rho }}%
\right) ^{2}\int_{\tilde{B}_{\rho _{i}}}\left( 1+\left( \left\vert
Du\right\vert -1\right) ^{\gamma _{i}}\mathcal{K}_M^{\tau }\left( \left\vert
Du\right\vert \right) \left\vert Du\right\vert ^{2\tau }\right) \,dx.
\end{equation*}%
By iterating (\ref{after_sobolev_31}), we get 
\begin{equation}
\left[ \int_{\tilde{B}_{\rho _{i+1}}}\left( 1+\left( \left\vert
Du\right\vert -1\right) ^{k_{i}\left( 2^{\ast }/2\right) ^{i+1}}\mathcal{K}%
_M^{\tau }\left( \left\vert Du\right\vert \right) \left\vert Du\right\vert
^{2\tau}\right) \,dx\right] ^{\left( 2/2^{\ast }\right) ^{i+1}}
\label{after_sobolev_32}
\end{equation}%
\begin{equation*}
\leq C\int_{\tilde{B}_{\overline{R}}}\left( 1+\left\vert Du\right\vert
^{2\tau}\mathcal{K}_M^{\tau }\left( \left\vert Du\right\vert \right) \right)
\,dx
\end{equation*}%
where we have denoted with $k_{i}=\left( 1-\beta -\frac{2}{2^{\ast }}\tau
\right) \left( 1-\frac{1}{(2^{\ast }/2)^{i+1}}\right) n$. The exponent in
the first integral is given by computing 
\begin{equation*}
\gamma _{i+1}=\gamma _{0}\left( \frac{2^{\ast }}{2}\right) ^{i+1}-2\left(
\beta +\frac{2}{2^{\ast }}\tau -1\right) \sum_{k=1}^{i+1}\left( \frac{%
2^{\ast }}{2}\right) ^{k}
\end{equation*}%
\begin{equation*}
=\left( 1-\beta -\frac{2}{2^{\ast }}\tau \right) \left[ \left( \frac{2^{\ast
}}{2}\right) ^{i+1}-1\right] n.
\end{equation*}%
Observe that the quantity $1-\beta -\frac{2}{2^{\ast }}\tau >0$ by the
restrictions on $\tau $. The constant $C$ in (\ref{after_sobolev_32}) is
such that 
\begin{equation*}
C\leq \prod_{k=0}^{\infty }\left( \frac{c(2^{\ast })^{3k}}{\overline{R}-%
\overline{\rho }}\right) ^{\left( \frac{2}{2^{\ast }}\right) ^{k}}=\left( 
\frac{c}{\left( \overline{R}-\overline{\rho }\right) ^{2}}\right)
^{\sum_{k=0}^{\infty }\left( \frac{2}{2^{\ast }}\right) ^{k}}\cdot (2^{\ast
})^{\sum_{k=0}^{\infty }k\left( \frac{2}{2^{\ast }}\right) ^{k}}
\end{equation*}%
\begin{equation*}
=\frac{c}{\left( \overline{R}-\overline{\rho }\right) ^{n}},
\end{equation*}%
for every $n\geq 3$; otherwise, if $n=2$, then for every $\varepsilon >0$ we
can choose $2^{\ast }>2$ so that $C=\frac{C}{(\overline{R}-\overline{\rho }%
)^{2+\varepsilon }}$.

We observe that the function $1+(t-1)^{\alpha}\mathcal{K}_M^{\tau}(t)t^{2%
\tau}\geq 1+(t-1)^{\alpha}t^{\tau}\left(h^{\prime}(t)\right)^{\tau}$, since $%
\mathcal{K}_M(t)\geq\frac{h^{\prime}(t)}{t}$ for every $t\geq 1$. Moreover,
since $h^{\prime}$ is increasing and $\tau\geq 1$ we have that $%
1+(t-1)^{\alpha}t^{\tau}\left(h^{\prime}(t)\right)^{\tau}\geq
1+(t-1)^{\alpha+\tau}\left(h^{\prime}(1)\right)^{\tau}$. Then, for every $%
i\geq 0$ we have 
\begin{equation}  \label{after_sobolev_33}
\left[\int_{\tilde{B}_{\rho_{i+1}}}\left( \left\vert
Du\right\vert-1\right)^{\left( k_i+\frac{\tau}{\left(2^{\ast}/2\right)^{i+1}}%
\right)\left(2^{\ast}/2\right)^{i+1}} \,dx \right]^{\left(2/2^{\ast}%
\right)^{i+1}}
\end{equation}
\begin{equation*}
\leq C \int_{\tilde{B}_{\overline{R}}} \left( 1+ \left\vert
Du\right\vert^{2\tau} \mathcal{K}_M^{\tau}\left( \left\vert Du\right\vert
\right)\right)\, dx.
\end{equation*}
Finally we go to the limit as $i\to\infty$ and we obtain 
\begin{equation*}
\sup\left\{ \left( \left\vert Du(x)\right\vert-1\right)^{\left( 1-\beta-%
\frac{2}{2^{\ast}}\tau\right)n }:\, x\in \tilde{B}_{\bar{\rho}}\right\}
\end{equation*}
\begin{equation*}
=\lim_{i\to\infty} \left[\int_{\tilde{B}_{\rho_{i+1}}}\left( \left\vert
Du\right\vert-1\right)^{\left( k_i+\frac{\tau}{\left(2^{\ast}/2\right)^{i+1}}%
\right)\left(2^{\ast}/2\right)^{i+1}} \,dx \right]^{\left(2/2^{\ast}%
\right)^{i+1}}
\end{equation*}
\begin{equation*}
\leq \frac{C}{\left(\bar{R}-\bar{\rho}\right)^n}\int_{\tilde{B}_{\bar{R}%
}}\left( 1+ \left\vert Du \right\vert^{2\tau}\mathcal{K}_M^{\tau}\left(\left%
\vert Du\right\vert \right)\right)\,dx.
\end{equation*}
\end{proof}

\begin{lemma}
\label{lemma2} Let $g, h$ respectively satisfy (\ref{main assumptions}) and (%
\ref{the function h(t)}). Suppose that the supplementary condition (\ref%
{suppl_ass}) is satisfied. Let $u\in W_{\limfunc{loc}}^{1,1}\left( \Omega ,%
\mathbb{R}^{m}\right) $ be a local minimizer of (\ref{energy-integral 1}).
Then, for every $\varepsilon >0$ and for every $\rho ,R$ $(0<\rho <R)$ there
exists a constant $C$ such that 
\begin{equation}
\int_{\tilde{B}_{\rho }}\left( 1+\left\vert Du\right\vert ^{2\tau }\mathcal{K%
}_{M}^{\tau }\left( \left\vert Du\right\vert \right) \right) \,dx\leq C\left[
\int_{\tilde{B}_{R}}\left( 1+h\left( \left\vert Du\right\vert \right)
\right) \,dx\right] ^{\frac{\tau }{1-\beta }+\varepsilon }.  \label{lemma2_1}
\end{equation}%
The constant $C$ depends on $n,\varepsilon ,\vartheta ,\rho ,R,\beta,\alpha$
and $\sup \left\{ h^{\prime \prime }(t):\;t\in \left[ 0,1\right] \right\}$.
\end{lemma}

\begin{proof}
In Lemma \ref{lemma for phi for any gamma} we have considered parameters $%
\alpha$ and $\gamma$ such that $\alpha\in\left( 1, \frac{2n}{2n-1}\right]$
and $\gamma\geq 0$. Here we restrict ourselves to the case $1<\alpha\leq%
\frac{2n\tau}{n(1+\tau)-1}$ and $\gamma=0$. Then, Lemma \ref{lemma for phi
for any gamma} holds for any $\nu\in \left[1, \frac{2^{\ast}(2-\alpha)}{%
2\alpha}\right]$. Since $\tau< \frac{2^{\ast}}{2}(1-\beta)$, we have that $%
1<(1-\beta )\frac{2^{\ast}}{2\tau}$, therefore it is possible to limit $\nu$
to satisfy the condition $1<\nu< (1-\beta) \frac{2^{\ast}}{2\tau}$. Finally,
since $\beta>\frac{1}{n} $, we have $\alpha\leq\frac{2n\tau}{n(1+\tau)-1}<%
\frac{2\tau}{1-\beta+\tau}$ which implies $1-\beta<\frac{2-\alpha}{\alpha}%
\tau$. Thus, 
\begin{equation*}
\nu\in \left[ 1, (1-\beta)\frac{2^{\ast}}{2\tau}\right]\subseteq\left[1,
2^{\ast}\frac{2-\alpha}{2\alpha}\right]
\end{equation*}
whence we obtain that the conditions of Lemma \ref{lemma for phi for any
gamma} are satisfied. Therefore there exists a constant $c$ such that 
\begin{equation*}
\left[G(t)\right]^{2^{\ast}}=\left[ \left( 1+\int_0^t \sqrt{\Phi(s)\mathcal{K%
}_m(s)}\,ds\right)^{\frac{2^{\ast}}{\nu}}\right]^{\nu}\geq c\left( 1+\left(
(t-1) t^{-\beta} \right)^{\frac{2^{\ast}}{\nu}}\mathcal{K}%
_M^{\tau}(t)\right)^{\nu}.
\end{equation*}
Under the notations of lemma \ref{lemma1}, formula (\ref{after_sobolev_3})
with $\gamma=0$ becomes

\begin{equation}  \label{lemma2_2}
\left[\int_{\tilde{B}_R}\eta^{2^{\ast}} \left(1+\left( \left\vert
Du\right\vert-1\right)^{ \left( 2^{\ast}-(2\tau+2^{\ast}\beta) \right)\frac{1%
}{\nu}}\mathcal{K}_M^{\tau}\left(\left\vert Du\right\vert\right) \left\vert
Du\right\vert^{\frac{2\tau}{\nu}} \right)^{\nu}\,dx \right]^{\frac{2}{%
2^{\ast}}}
\end{equation}
\begin{equation*}
\leq c\int_{\tilde{B}_R} \left( \eta^2+\left\vert D\eta\right\vert^2\right)
\left( 1 +\mathcal{K}_M^{\tau}\left(\left\vert
Du\right\vert\right)\left\vert Du\right\vert^{2\tau} \right)\,dx.
\end{equation*}
Moreover, since there exists a constant $C_1$ such that 
\begin{equation*}
\left[\int_{\tilde{B}_R}\eta^{2^{\ast}} \left(1+\left( \left\vert
Du\right\vert-1\right)^{ \left( 2^{\ast}-(2\tau+2^{\ast}\beta) \right)\frac{1%
}{\nu}}\mathcal{K}_M^{\tau}\left(\left\vert Du\right\vert\right) \left\vert
Du\right\vert^{\frac{2\tau}{\nu}} \right)^{\nu}\,dx \right]^{\frac{2}{%
2^{\ast}}}
\end{equation*}
\begin{equation*}
\geq C_1\left[\int_{\tilde{B}_R}\eta^{2^{\ast}} \left(1+ \left\vert
Du\right\vert^{ 2^{\ast}\frac{1-\beta}{\nu}}\mathcal{K}_M^{\tau}\left(\left%
\vert Du\right\vert\right) \right)^{\nu}\,dx \right]^{\frac{2}{2^{\ast}}},
\end{equation*}
(\ref{lemma2_2}) gives 
\begin{equation}  \label{lemma2_21}
\left[\int_{\tilde{B}_R}\eta^{2^{\ast}} \left(1+ \left\vert Du\right\vert^{
2^{\ast}\frac{1-\beta}{\nu}}\mathcal{K}_M^{\tau}\left(\left\vert
Du\right\vert\right) \right)^{\nu}\,dx \right]^{\frac{2}{2^{\ast}}}
\end{equation}
\begin{equation*}
\leq c\int_{\tilde{B}_R} \left( \eta^2+\left\vert D\eta\right\vert^2\right)
\left( 1 +\mathcal{K}_M^{\tau}\left(\left\vert
Du\right\vert\right)\left\vert Du\right\vert^{2\tau} \right)\,dx.
\end{equation*}

Since $\nu<(1-\beta)\frac{2^{\ast}}{2\tau}$, we have $2^{\ast}\frac{1-\beta}{%
\nu} >2\tau$ and then, if we define $V=V(x)=1+\left\vert
Du\right\vert^{2\tau}\mathcal{K}_M^{\tau}\left( \left\vert Du\right\vert
\right)$, (\ref{lemma2_21}) becomes

\begin{equation}  \label{lemma2_2_2}
\left[\int_{\tilde{B}_R}\eta^{2^{\ast}} V^{\nu}\,dx \right]^{\frac{2}{%
2^{\ast}}}\leq c\int_{\tilde{B}_R} \left( \eta^2+\left\vert
D\eta\right\vert^2\right)V \,dx.
\end{equation}

As in the previous Lemma \ref{lemma1}, we consider a test function $\eta$
equal to $1$ in $\tilde{B}_{\rho}$ with $\left\vert D\eta\right\vert\leq%
\frac{2}{R-\rho}$ and we obtain 
\begin{equation}  \label{lemma2_2_3}
\left(\int_{\tilde{B}_{\rho}}V^{\nu}\,dx\right)^{\frac{2}{2^{\ast}}}\leq%
\frac{c}{(R-\rho)^2}\int_{\tilde{B}_R} V\,dx.
\end{equation}
Let $\mu>\frac{2^{\ast}}{2}$, then by the H\"older inequality we have 
\begin{equation}  \label{lemma2_2_4}
\left( \int_{\tilde{B}_{\rho}}V^{\nu}\,dx\right)^{\frac{2}{2^{\ast}}}\leq 
\frac{c}{(R-\rho)^2}\int_{\tilde{B}_{R}} V^{\frac{\nu}{\mu}}V^{1-\frac{\nu}{%
\mu}}\,dx
\end{equation}
\begin{equation*}
\leq\left( \int_{\tilde{B}_{R}} V^{\nu}\,dx\right)^{\frac{1}{\mu}}\left(
\int_{\tilde{B}_R}V^{\frac{\mu-\nu}{\mu-1}}\,dx \right)^{\frac{\mu-1}{\mu}}.
\end{equation*}
Let $R_0$ and $\rho_0$ be fixed. For any $i\in\mathbb{N}$ we consider (\ref%
{lemma2_2_4}) with $R=\rho_i$ and $\rho=\rho_{i-1}$, where $\rho_i=R_0-\frac{%
R_0-\rho_0}{2^i}$. By iterating (\ref{lemma2_2_4}), since $R-\rho=\frac{%
R_0-\rho_0}{2^i}$, similarly to the computation in \cite{Marcellini 1996}
and \cite{Marcellini-Papi 2006} we can write 
\begin{equation}  \label{lemma2_2_5}
\int_{\tilde{B}_{\rho_0}}V^{\nu}\,dx \leq c \left(\frac{1}{(R_0-\rho_0)^2}%
\right)^{\left( \frac{2^{\ast}\mu }{ 2\mu-2 }\right)^i} \left( \int_{\tilde{B%
}_{\rho_i}}V^{\nu}\,dx\right)^{\left( \frac{2^{\ast}}{2\mu}\right)^i}
\end{equation}
\begin{equation*}
\cdot \left( \int_{\tilde{B}_{\rho_0}}V^{\frac{\mu-\nu}{\mu-1}}\,dx \right)^{%
\frac{2^{\ast}(\mu-1)}{2\mu-2^{\ast}}}.
\end{equation*}
Since $\frac{\mu-\nu}{\mu-1}<1$ we can apply Lemma \ref{K_lemma} and obtain 
\begin{equation*}
\int_{\tilde{B}_{\rho_0}}V^{\nu}\,dx \leq c \left(\frac{1}{(R_0-\rho_0)^2}%
\right)^{ \frac{2^{\ast}\mu }{ 2\mu-2^{\ast} }} \left( \int_{B_{\rho_i}}V^{%
\frac{1}{\sigma}}\,dx\right)^{\left( \frac{2^{\ast}}{2\mu}\right)^i}
\end{equation*}
\begin{equation*}
\cdot \left( \int_{\tilde{B}_{\rho_0}} \left[ 1+h\left(\left\vert
Du\right\vert \right)\right] \,dx \right)^{\frac{2^{\ast}(\mu-1)}{%
2\mu-2^{\ast}}}.
\end{equation*}
In the limit as $i\to\infty$ we get 
\begin{equation*}
\int_{\tilde{B}_{\rho_0}}V^{\nu}\,dx \leq c\left(\frac{1}{(R_0-\rho_0)^2}%
\right)^{ \frac{2^{\ast}\mu }{ 2\mu-2^{\ast} }} \left(\int_{\tilde{B}%
_{\rho_0}} \left[ 1+h\left(\left\vert Du\right\vert \right)\right] \,dx
\right)^{\frac{2^{\ast}(\mu-1)}{2\mu-2^{\ast}}}.
\end{equation*}
Finally, 
\begin{equation}  \label{lemma2_2_6}
\int_{B_{\rho_0}} V\, dx \leq \vert \tilde{B}_{\rho_0}\vert^{1-\frac{1}{\nu}%
} \left( \int_{\tilde{B}_{\rho_0}} V^{\nu}\,dx\right)^{\frac{1}{\nu}}
\end{equation}
\begin{equation*}
\leq c \left(\frac{1}{(R_0-\rho_0)^2}\right)^{ \frac{2^{\ast}\mu }{
(2\mu-2^{\ast})\nu }} \left(\int_{\tilde{B}_{R_0}} \left[ 1+h\left(\left%
\vert Du\right\vert \right)\right] \,dx \right)^{\frac{2^{\ast}(\mu-1)}{%
(2\mu-2^{\ast})\nu}}.
\end{equation*}
As $\nu\to \frac{2^{\ast}(1-\beta)}{2\tau}$ and $\mu\to\infty$ the two
exponents in (\ref{lemma2_2_6}) converge to $\frac{\tau}{1-\beta}$ and we
have the result.
\end{proof}

By combining Lemmata \ref{lemma1}, \ref{lemma2} and by using again (\ref%
{main assumptions}), we obtain

\begin{equation}
\left\Vert Du\right\Vert _{L^{\infty }\left( \tilde{B}_{\rho },\mathbb{R}%
^{m\times n}\right) }^{\left( 1-\beta -\frac{2}{2^{\ast }}\tau \right)
n}\leq C\int_{\tilde{B}_{R}}\left( 1+\left\vert Du\right\vert ^{2\tau }%
\mathcal{K}_{M}^{\tau }\left( \left\vert Du\right\vert \right) \right) \,dx
\label{finale_1}
\end{equation}%
\begin{equation*}
\leq C^{\prime }\left[ \int_{\tilde{B}_{R}}\left( 1+h\left( \left\vert
Du\right\vert \right) \right) \,dx\right] ^{\frac{\tau }{1-\beta }%
+\varepsilon }\leq C^{\prime }\left[ \int_{\tilde{B}_{R}}\left( 1+g\left(
x,\left\vert Du\right\vert \right) \right) \,dx\right] ^{\frac{\tau }{%
1-\beta }+\varepsilon },
\end{equation*}%
since, by (\ref{main assumptions}) and the fact that $h(0)=0=g(x,0)$, $%
h\left( t\right) =\int_{0}^{t}h^{\prime }\left( s\right) \,ds\leq
\int_{0}^{t}g_{t}\left( x,s\right) \,ds=g\left( x,t\right) $. In order to go
from $\tilde{B}_{\rho },\tilde{B}_{R}$ to $B_{\rho },B_{R}$ we observe that $%
\left\Vert Du\right\Vert _{L^{\infty }\left( \tilde{B}_{\rho },\mathbb{R}%
^{m\times n}\right) }\leq 1+\left\Vert Du\right\Vert _{L^{\infty }\left(
B_{\rho },\mathbb{R}^{m\times n}\right) }$ and from (\ref{finale_1})\ we
also get%
\begin{equation*}
\left\Vert Du\right\Vert _{L^{\infty }\left( B_{\rho },\mathbb{R}^{m\times
n}\right) }^{\left( 1-\beta -\frac{2}{2^{\ast }}\tau \right) n}\leq
C^{\prime \prime }\left[ \int_{B_{R}}\left( 1+g\left( x,\left\vert
Du\right\vert \right) \right) \,dx\right] ^{\frac{\tau }{1-\beta }%
+\varepsilon }.
\end{equation*}
We summarize in the next statement the a-priori estimate that we have proved.

\begin{theorem}
\label{thm_1} Let $g,h$ respectively\ satisfy (\ref{main assumptions}) and (%
\ref{the function h(t)}). Suppose that the supplementary condition (\ref%
{suppl_ass}) is satisfied. Let $u\in W_{\limfunc{loc}}^{1,1}\left( \Omega ,%
\mathbb{R}^{m}\right) $ be a local minimizer of (\ref{energy-integral 1}).
Then, for every $\varepsilon >0$ and for every $\rho ,R$ $(0<\rho <R)$,
there exists a constant $C$ such that 
\begin{equation}
\left\Vert Du\right\Vert _{L^{\infty }\left( B_{\rho },\mathbb{R}^{m\times
n}\right) }^{\left( 1-\beta -\frac{2}{2^{\ast }}\tau \right) n}\leq C\left[
\int_{B_{R}}\left( 1+g\left( x,\left\vert Du\right\vert \right) \right) \,dx%
\right] ^{\frac{\tau }{1-\beta }+\varepsilon },
\label{conclusion in the main theorem finale}
\end{equation}%
where $\tau =(2\vartheta -1)\vartheta $.\ The constant $C$ depends on $%
n,\varepsilon ,\vartheta ,\rho ,R,t_{0},\beta ,\alpha $ and $\sup \bigl\{
h^{\prime \prime }(t):$ $\;t\in \left[ 0,t_{0}\right] \bigr\} $, but is
independent of the constants in supplementary condition (\ref{suppl_ass}).
\end{theorem}

\bigskip

We note that $1-\beta -\frac{2}{2^{\ast }}\tau >0$ since $\tau <\left(
1-\beta \right) \frac{2^{\ast }}{2}$. Note also that $\left( 1-\beta -\frac{2%
}{2^{\ast }}\tau \right) n<1$; in fact, since $\tau \geq 1$ and $\beta >%
\frac{1}{n}$,%
\begin{equation*}
\left( 1-\beta -\frac{2}{2^{\ast }}\tau \right) n<\left( 1-\frac{1}{n}-\frac{%
2}{2^{\ast }}\right) n=1.
\end{equation*}%
Moreover $\frac{\tau }{1-\beta }+\varepsilon >1$ and thus (\ref{conclusion
in the main theorem finale}) gives the final representation of the a-priori
estimate as stated in (\ref{conclusion in the main theorem}).


\begin{thebibliography}{99}
\bibitem{Baroni-Colombo-Mingione 2015} \text{P. Baroni, M. Colombo, G.
Mingione,} \textit{Harnack inequalities for double phase functionals}, {%
Nonlinear Analysis,} {121}{\ (2015), 206-222.}

\bibitem{Baroni-Colombo-Mingione 2016} \text{P. Baroni, M. Colombo, G.
Mingione,} \textit{Nonautonomous functionals, borderline cases and related
function classes}, Algebra i Analiz., 27 (2015), 6--50; translation in St.
Petersburg Math. J. 27 (2016), 347-379.

\bibitem{Baroni-Colombo-Mingione 2018} \text{P. Baroni, M. Colombo, G.
Mingione,} \textit{Regularity for general functionals with double phase},
Calc. Var. Partial Differential Equations, 57 (2018), 57-62.

\bibitem{Beck-Mingione 2019} L. Beck, G. Mingione, \textit{Lipschitz bounds
and non-uniformly ellipticity}, Communications on Pure and Applied Math., to
appear, 2018.

\bibitem{Bousquet-Brasco 2019} P. Bousquet, L. Brasco, \textit{Lipschitz
regularity for orthotropic functionals with nonstandard growth conditions},
preprint 2018.

\bibitem{Carozza-Giannetti-Leonetti-Passarelli 2018} M. Carozza, F.
Giannetti, F. Leonetti, A. Passarelli di Napoli, \textit{Pointwise bounds
for minimizers of some anisotropic functionals}, Nonlinear Anal., 177
(2018), 254-269.

\bibitem{Cencelja-Radulescu-Repovs 2018} M. Cencelja, V. R\u{a}dulescu, D.
Repov\v{s}, \textit{Double phase problems with variable growth}, {Nonlinear
Analysis,} {(2018), to appear.}

\bibitem{Chlebicka 2018} I. Chlebicka, \textit{A pocket guide to nonlinear
differential equations in Musielak--Orlicz spaces}, Nonlinear Analysis, 175
(1918) 1-27.

\bibitem{Chlebicka 2019} I. Chlebicka, P. Gwiazda, A. Zatorska--Goldstein, 
\textit{Parabolic equation in time and space dependent anisotropic
Musielak--Orlicz spaces in absence of Lavrentiev's phenomenon}, Ann.
I.H.Poincar\'{e}, 36 (2019), 1431-1465.

\bibitem{Colombo-Mingione 2015a} \text{M. Colombo, G. Mingione,} \textit{%
Regularity for double phase variational problems}\emph{,} {Arch. Rat. Mech.
Anal.},\text{\ 215}{\ (2015), 443-496.}

\bibitem{Colombo-Mingione 2015b} \text{M. Colombo, G. Mingione,} \textit{%
Bounded minimisers of double phase variational integrals},{\ Arch. Rat.
Mech. Anal.},\text{\ 218}{\ (2015), 219-273.}

\bibitem{Cupini-Giannetti-Giova-Passarelli 2018} G. Cupini, F. Giannetti, R.
Giova, A. Passarelli di Napoli, \textit{Regularity results for vectorial
minimizers of a class of degenerate convex integrals}, J. Differential
Equations, 265 (2018), 4375-4416.

\bibitem{Cupini-Marcellini-Mascolo 2009} \text{G. Cupini, P. Marcellini, E.
Mascolo,} \textit{Regularity under sharp anisotropic general growth
conditions}, Discrete Contin. Dyn. Syst. Ser. B, 11 (2009), 66-86.

\bibitem{Cupini-Marcellini-Mascolo 2012} \text{G. Cupini, P. Marcellini, E.
Mascolo, }\textit{Local boundedness of solutions to quasilinear elliptic
systems}, Manuscripta Math., 137 (2012), 287-315.

\bibitem{Cupini-Marcellini-Mascolo 2018} \text{G. Cupini, P. Marcellini, E.
Mascolo, }\textit{Nonuniformly elliptic energy integrals with }$p,q-$\textit{%
growth},{\ }Nonlinear Analysis, 177 (2018), 312-324.

\bibitem{De Giorgi 1968} E. De Giorgi, \textit{Un esempio di estremali
discontinue per un problema variazionale di tipo ellittico}, (Italian) Boll.
Un. Mat. Ital., 1 (1968), 135-137.

\bibitem{DeFilippis 2018} C. De Filippis, \textit{Higher integrability for
constrained minimizers of integral functionals with }$(p,q)-$\textit{growth
in low dimension}, {\ Nonlinear Anal.,} 170{\ (2018), 1-20}.

\bibitem{DeFilippis-Mingione 2019} C. De Filippis, G. Mingione, \textit{On
the regularity of minima of non-autonomous functionals}, J. Geometric
Analysis, (2019), to appear.

\bibitem{DeSilva-Savin 2010} D. De Silva, O. Savin, \textit{Minimizers of
convex functionals arising in random surfaces}, Duke Math. J., 151 (2010),
487-532.

\bibitem{Eleuteri-Marcellini-Mascolo 2016a} \text{M. Eleuteri, P.
Marcellini, E. Mascolo,} \textit{Lipschitz estimates for systems with
ellipticity conditions at infinity}, {Ann. Mat. Pura e Appl.,} 195 (2016),
1575-1603.

\bibitem{Eleuteri-Marcellini-Mascolo 2016b} \text{M. Eleuteri, P.
Marcellini, E. Mascolo,} \textit{Lipschitz continuity for energy integrals
with variable exponents}, Atti Accad. Naz. Lincei Rend. Lincei Mat. Appl. 27
(2016), 61-87.

\bibitem{Eleuteri-Marcellini-Mascolo 2018} \text{M. Eleuteri, P. Marcellini,
E. Mascolo,} \textit{Regularity for scalar integrals without structure
conditions}, {Adv. Calc. Var.} (2018), in press, DOI:
https://doi.org/10.1515/acv-2017-0037.

\bibitem{Esposito-Leonetti-Mingione 2004} L. Esposito, F. Leonetti, G.
Mingione, \textit{Sharp regularity for functionals with }$(p,q)$\textit{\
growth}, J. Differential Equations, 204 (2004), 5-55.

\bibitem{Giusti-Miranda 1968} E. Giusti, M. Miranda, \textit{Un esempio di
soluzioni discontinue per un problema di minimo relativo ad un integrale
regolare del calcolo delle variazioni}, (Italian) Boll. Un. Mat. Ital., 1
(1968), 219-226.

\bibitem{Harjulehto-Hasto-Toivanen 2017} P. Harjulehto, P. H\"{a}st\"{o},
O.Toivanen, \textit{H\"{o}lder regularity of quasiminimizers under
generalized growth conditions}, Calc. Var. Partial Differential Equations 56
(2017), 26 pp.

\bibitem{Hasto-Ok 2019} P. H\"{a}st\"{o}, J. Ok, \textit{Maximal regularity
for local minimizers of non-autonomous functionals}, Preprint (2019).

\bibitem{Marcellini 2019} P. Marcellini, \textit{A variational approach to
parabolic equations under general and }$p,q-$\textit{growth conditions},
Nonlinear Analysis, (2019). https://doi.org/10.1016/j.na.2019.02.010

\bibitem{Marcellini 1996} P. Marcellini, \textit{Everywhere regularity for a
class of elliptic systems without growth conditions}, Ann. Scuola Norm. Sup.
Pisa Cl. Sci., 23 (1996), 1-25.

\bibitem{Marcellini DiscrContDinSystems 2019} P. Marcellini, \textit{%
Regularity under general and }$p,q-$\textit{growth conditions}, Discrete
Cont. Din. Systems, S Series, (2019), to appear.

\bibitem{Marcellini-Papi 2006} P. Marcellini, G. Papi, \textit{Nonlinear
elliptic systems with general growth}, J. Differential Equations, 221
(2006), 412-443.

\bibitem{Mascolo-Migliorini 2003} E. Mascolo, A. Migliorini, \textit{%
Everywhere regularity for vectorial functionals with general growth}, ESAIM
Control Optim. Calc. Var., 9 (2003), 399-418.

\bibitem{Mingione-Palatucci 2019} G. Mingione, G. Palatucci, Developments
and perspectives in Nonlinear Potential Theory, Nonlinear Analysis, (2019),
to appear.

\bibitem{Mooney-Savin 2016} C. Mooney, O. Savin, \textit{Some singular
minimizers in low dimensions in the calculus of variations}, Arch. Ration.
Mech. Anal., 221 (2016), 1-22.

\bibitem{Mooney 2019} C. Mooney, \textit{Minimizers of convex functionals
with small degeneracy set}, preprint 2019.

\bibitem{Necas 1977} J. Ne\v{c}as, \textit{Example of an irregular solution
to a nonlinear elliptic system with analytic coefficients and conditions for
regularity}. Theory of nonlinear operators (Proc. Fourth Internat. Summer
School, Acad. Sci., Berlin, 1975), 197--206. Abh. Akad. Wiss. DDR Abt.
Math.-Natur.-Tech., Jahrgang 1977, 1, Akademie-Verlag, Berlin, 1977.

\bibitem{Radulescu-Zhang 2018} \text{V. R\v{a}dulescu, Q. Zhang,} \textit{%
Double phase anisotropic variational problems and combined effects of
reaction and absorption terms}, {J. Math. Pures Appl.,} 118 (2018), 159-203.

\bibitem{Sverak-Yan 2000} V. \v{S}ver\'{a}k, X. Yan, \textit{A singular
minimizer of a smooth strongly convex functional in three dimensions}, Calc.
Var. Partial Differential Equations, 10 (2000), 213-221.

\bibitem{Uhlenbeck 1977} K. Uhlenbeck, \textit{Regularity for a class of
non-linear elliptic systems}, Acta Math., 138 (1977), 219-240.
\end{thebibliography}
\end{document}